\documentclass[preprint,12pt]{elsarticle}
\pdfoutput=1

\usepackage[utf8]{inputenc}
\usepackage[T1]{fontenc}
\IfFileExists{lmodern.sty}{\usepackage{lmodern}}{}  

\usepackage{amsmath,amssymb,amsfonts}
\usepackage{amsthm}
\usepackage{mathtools}

\newtheorem{theorem}{Theorem}[section]
\newtheorem{proposition}[theorem]{Proposition}
\newtheorem{lemma}[theorem]{Lemma}
\newtheorem{corollary}[theorem]{Corollary}
\theoremstyle{definition}
\newtheorem{definition}[theorem]{Definition}
\newtheorem{example}[theorem]{Example}
\theoremstyle{remark}
\newtheorem{remark}[theorem]{Remark}

\usepackage{booktabs}
\usepackage{array}
\usepackage{enumitem}

\usepackage{algorithm,algpseudocode}

\usepackage[final,expansion=false]{microtype}
\usepackage{xurl}
\usepackage[colorlinks=true,
linkcolor={blue!70!black},
citecolor={blue!70!black},
urlcolor={blue!80!black}]{hyperref}
\usepackage{xcolor}

\hypersetup{%
	pdftitle={Exact Stratification and Affine Mass Formulas for Split Richelot Data over Finite Fields},
	pdfauthor={Hung T. Dang, Diep V. Nguyen}}

\newcommand{\uptounit}{\mathrel{\sim_{\mathrm{unit}}}}
\DeclareMathOperator{\lc}{lc}
\DeclareMathOperator{\ct}{ct}
\DeclareMathOperator{\Res}{Res}
\DeclareMathOperator{\Disc}{Disc}
\DeclareMathOperator{\Jac}{Jac}
\DeclareMathOperator{\Bez}{Bez}
\DeclareMathOperator{\Stab}{Stab}
\providecommand{\PGL}{\operatorname{PGL}}

\newcommand{\F}{\mathbb{F}}
\newcommand{\Fq}{\mathbb{F}_q}
\newcommand{\Z}{\mathbb{Z}}
\newcommand{\PP}{\mathbb{P}}
\newcommand{\A}{\mathbb{A}}

\newcommand{\dcoef}{\mathfrak{d}}      
\newcommand{\Aff}{\operatorname{Aff}}
\newcommand{\Fix}{\operatorname{Fix}}
\newcommand{\mass}{\operatorname{mass}}
\newcommand{\Adm}{\mathcal{A}}         
\newcommand{\Dec}{\mathcal{D}}         
\newcommand{\Gfive}{\mathcal{G}_5}
\newcommand{\Gsix}{\mathcal{G}_6}

\journal{Finite Fields and Their Applications}
\biboptions{numbers,sort&compress}

\begin{document}
	
	\begin{frontmatter}
		
		\title{Exact Stratification and Affine Mass Formulas for Split Richelot Data over Finite Fields}
		
		\author[pdu]{Hung T. Dang\corref{cor1}\fnref{fn:orcid}}
		\ead{hung.dt@phuongdong.edu.vn}
		\fntext[fn:orcid]{ORCID: \href{https://orcid.org/0009-0006-3272-0573}{0009-0006-3272-0573}}
		\author[pdu]{Diep V. Nguyen}
		\cortext[cor1]{Corresponding author.}
		\affiliation[pdu]{organization={Department of Mathematics, University of Phuong Dong},
			city={Hanoi}, country={Vietnam}}
		
		\begin{abstract}
			The Richelot \((2,2)\)-step is the standard step of explicit genus-2 isogeny computation. We determine the exact stratification of its input space over a
			finite field \(\Fq\) of odd characteristic: ordered factorizations \(f=uvw\)
			of a square-free sextic into monic quadratics fall into three strata by the
			geometric type of the quotient, governed by the incidence geometry of the
			discriminant locus. This yields closed formulas for each stratum and, modulo
			affine coordinate changes, mass formulas of degree four in \(q\) with a
			complete classification of stabilisers. The classification is decided by data the step already computes, at \(5\mathrm{M}+6\mathrm{S}\) beyond the brackets, and an output post-check is provably redundant. On the decomposable stratum the square class of one resultant determines the field of definition of the
			elliptic factors and the shape of the Weil polynomial of the Jacobian; the two cases are counted exactly, and in the nonsplit case the curve \(y^{2}=f\) has \(q+1\) rational points. Exhaustive enumeration over small finite fields verifies every proved count.
		\end{abstract}
		
		
	\end{frontmatter}
	
	\section{Introduction}\label{sec:intro}
	
	The Richelot $(2,2)$-isogeny is the fundamental construction for explicit genus-2 isogeny computation. 
	It starts from a square-free sextic presented as a product of three
	monic quadratics, \(f=uvw\) with \(u,v,w\in\Fq[x]\) and \(q\) odd, forms the
	brackets
	\[
	U=v'w-vw',\qquad V=w'u-wu',\qquad W=u'v-uv',
	\]
	and returns the curve cut out by \(\dcoef y^{2}=UVW\), where \(\dcoef\) is the
	determinant of the coefficient matrix of \((u,v,w)\) \cite{CasselsFlynn1996}.
	Throughout, \(C\colon y^{2}=f\) denotes the domain curve. Three outcomes are
	possible. The product \(UVW\) can be square-free of degree
	six; it can be square-free of degree five, which happens exactly when two of the
	middle coefficients agree; or \(\dcoef\) can vanish, in which case the three
	quadratics span a pencil, the brackets are pairwise proportional, and the
	quotient abelian surface is a product of elliptic curves. The classical construction does not describe how the space of inputs distributes over these three outcomes, nor does it determine the arithmetic of the third over the base field.
	
	The organizing observation is that every predicate involved is a statement of
	plane incidence geometry. Identify the monic quadratic \(x^{2}+ax+b\) with the
	point \((a,b)\in\A^{2}(\Fq)\) and let \(\Pi\colon b=a^{2}/4\) be the discriminant
	parabola. A quadratic is separable exactly when its point avoids \(\Pi\). Two
	separable quadratics share a root exactly when their points lie on a common
	tangent line of \(\Pi\). A pair is aligned, so that its bracket drops to degree
	one, exactly when the line joining the two points is vertical. And
	\(\dcoef=0\) exactly when the three points are collinear. This is
	Lemma~\ref{lem:dictionary} below.
	
	The dictionary converts the stratification problem into an incidence count.
	Every line of \(\A^{2}(\Fq)\) is tangent, vertical, secant or external to
	\(\Pi\), and carries \(q-1\), \(q-1\), \(q-2\) or \(q\) separable points
	respectively. Counting incidences gives the exact cardinality of the admissible
	locus \(\Adm\) and of its three strata (Theorem~\ref{thm:counts}): the
	decomposable stratum \(\Dec\) where \(\dcoef=0\), the stratum \(\Gfive\) of
	quintic genus-2 output, and the stratum \(\Gsix\) of sextic genus-2 output. The
	densities of \(\Dec\) and \(\Gfive\) that follow are \(1/q+O(q^{-3})\) and
	\(3/q+O(q^{-2})\).
	
	A displayed splitting is a presentation, not an intrinsic object, so the counts
	of Theorem~\ref{thm:counts} depend on the coordinate in which \(f\) is written.
	We therefore quotient by the group \(\Aff(\Fq)\) of affine coordinate changes
	\(x\mapsto ax+b\) acting on displayed splittings, and record the weighted count
	\(\sum_{\mathcal O}1/\#\Stab(\mathcal O)\), which we call the affine groupoid
	mass. The masses are polynomials of degree four in \(q\)
	(Theorem~\ref{thm:mass}):
	\[
	\mass_{\Aff}(\Adm)=(q-2)^{2}(q^{2}-q+4),\quad
	\mass_{\Aff}(\Dec)=(q-2)(q^{2}-3q+3),
	\]
	\[
	\mass_{\Aff}(\Gfive)=3(q^{3}-6q^{2}+14q-13),\quad
	\mass_{\Aff}(\Gsix)=q^{4}-9q^{3}+35q^{2}-71q+61 .
	\]
	The two degrees removed by the group are exactly the coordinate freedom in the
	displayed model, and the fourth mass is the residual of the first three. The action is not free: the only nontrivial stabilisers are the \(q\) involutions
	\(x\mapsto-x+b\), there are exactly \(2(q-2)(q-3)\) orbits with a nontrivial
	stabiliser, and all of them lie in \(\Dec\). The converse fails: the stabilised locus is the part of \(\Dec\) carried by vertical lines, which is a proper sublocus (Proposition~\ref{prop:stab}).
	
	The classification is decided by data already computed by the step, so the certificate is a by-product rather than an additional pass. The bracket of a pair is the
	diagonal of its Bezoutian, whose three coefficients are the \(2\times2\) minors
	of the coefficient matrix of the pair; the pairwise resultant is
	\(M_1^{2}-M_2M_0\) in those same minors; and \(\dcoef\) is the minor syzygy
	evaluated on the third quadratic, which is two multiplications on monic input
	(Proposition~\ref{prop:d-syzygy}). The resulting evaluator
	(Algorithm~\ref{alg:richelot}, Theorem~\ref{thm:certificate}) is total on
	displayed monic quadratic triples, classifies each into one of
	\textsc{Invalid}, \textsc{Genus2} and \textsc{Decomposable}, and costs
	\(5\mathrm{M}+6\mathrm{S}\) beyond the \(6\mathrm{M}\) bracket computation, or
	\(2\mathrm{M}\) along an already certified chain. Verifying the output triple is
	provably unnecessary (Proposition~\ref{prop:postcheck}), and so is the affine
	retry that an earlier version of this work used (Proposition~\ref{prop:shift}).
	
	The stratum \(\dcoef=0\) is settled as arithmetic rather than rejected as
	failure. The three quadratics lie in one pencil, and the deck involution of that
	pencil lifts to an extra involution of \(C\)
	(Proposition~\ref{prop:extra-involution}). The square class of one resultant,
	which we call the splitting class, is independent of the pair and equals the
	square class of an invariant of the line joining the three points
	(Proposition~\ref{prop:class-line}); this splits the count of \(\Dec\) into two
	exact subcounts (Corollary~\ref{cor:count-class}). When the class is a square
	the two elliptic factors are individually \(\Fq\)-rational; when it is not, they
	are defined over \(\F_{q^{2}}\) and exchanged by Galois, the Weil polynomial of
	\(\Jac(C)\) is even, and \(\#C(\Fq)=q+1\), an implementation check that uses no
	elliptic data at all (Corollary~\ref{cor:pointcount}).
	
	Finally, the two densities one might attach to the decomposable branch do not
	agree, and the disagreement is recorded
	(\ref{app:superspecial}). On the parameter space, \(\dcoef=0\) is a
	hypersurface and carries a \(1/q+O(q^{-3})\) fraction of the admissible
	triples; this is a proved counting statement. On the superspecial locus over
	\(\F_{p^{2}}\), the setting of the genus-2 isogeny problem
	\cite{CostelloSmith2020}, published neighbour and vertex-type counts
	\cite[Prop.~4.3]{KatsuraTakashima2020}\cite[Table~2]{FloritSmith2022auto}
	indicate that decomposable quotients occur only at vertices carrying an extra
	involution, at an overall frequency of order \(1/p\) rather than
	\(1/p^{2}\). That comparison is conditional on the quoted external counts and
	is independent of the proved theorems of this paper.
	
	\paragraph{Contributions}
	\begin{enumerate}[label=(\arabic*),leftmargin=*]
		\item We prove an incidence dictionary translating separability, coprimality,
		bracket degree drop and decomposability into avoidance of the discriminant
		parabola, avoidance of a common tangent, verticality of the join, and
		collinearity
		(Lemma~\ref{lem:dictionary}), and use it to prove exact closed formulas for the
		admissible locus and each of its three strata for every odd prime power \(q\)
		(Theorem~\ref{thm:counts}).
		\item We compute the affine groupoid mass of each stratum, classify all
		stabilisers of the \(\Aff(\Fq)\)-action, count the orbits and the stabilised
		orbits, and locate the stabilised locus exactly inside the decomposable stratum
		(Theorem~\ref{thm:mass}, Proposition~\ref{prop:stab}).
		\item We give a certificate that classifies a Richelot step from the Bezoutian
		minors it already computes, prove that an output post-check and an affine retry
		are redundant, and state the cost as exact operation counts in a declared cost
		model (Theorem~\ref{thm:certificate}, Propositions~\ref{prop:postcheck}
		and~\ref{prop:shift}, \S\ref{sec:verify}).
		\item We identify the splitting class of a decomposable configuration with the
		square class of a line invariant, count the two cases exactly
		(Corollary~\ref{cor:count-class}), and determine the field of definition of the
		elliptic factors, the two shapes of the Weil polynomial, and the point-count
		identity \(\#C(\Fq)=q+1\) in the nonsplit case (\S\ref{sec:split}).
	\end{enumerate}
	
	\paragraph{What is classical and what is not}
	The soundness of the Richelot step (Proposition~\ref{prop:soundness}) is
	classical \cite{CasselsFlynn1996}. The Bezoutian of two polynomials, the
	identity between its determinant and the resultant, and the bracket as the
	Jacobian covariant of the pair are classical
	\cite[Ch.~4]{BasuPollackRoy2006}\cite{HelmkeFuhrmann1989}\cite{Olver1999}.
	Bolza \cite[\S11]{Bolza1887} introduced the invariant whose vanishing detects an
	extra involution of a binary sextic. The product model of a \((2,2)\)-split Jacobian through
	complementary degree-two elliptic subcovers is classical, going back to the
	double-cover geometry recorded by Kuhn \cite[\S1]{Kuhn1988} and by Smith
	\cite[\S8]{Smith2005}. Explicit formulas for decomposed Richelot
	isogenies are given by Castryck, Decru and Smith
	\cite{CastryckDecruSmith2020} and are used in the counting results of Katsura
	and Takashima \cite{KatsuraTakashima2020} and in the atlas of Florit and Smith
	\cite{FloritSmith2022atlas}; those formulas take a sextic whose six
	Weierstrass points are individually rational, and each step there costs three
	square roots. Detection of \((N,N)\)-splittings from moduli, through
	Igusa--Clebsch invariants and Humbert surfaces, is a complementary technique
	\cite{SantosCostelloFrengley2024}: it decides whether a Jacobian is split
	without producing a kernel or a quotient. The counting theorems, the affine
	groupoid mass with its stabiliser analysis, the identification of the splitting
	class with a line invariant together with the resulting refined count, and the
	base-field arithmetic of the decomposable branch are, to our knowledge, new.
	
	\paragraph{Relation to an earlier version}
	An earlier preprint \cite[v1]{DangNguyen2026v1} stated the soundness proposition
	without the hypothesis \(\dcoef\neq0\), used the un-normalized model
	\(y^{2}=UVW\), guarded the step with an output post-check and an affine retry,
	and reported a wall-clock speedup for the derivative-free bracket assembly. All
	four points are corrected here: the hypothesis is restored, the model is
	normalized by \(\dcoef\), the post-check and the retry are proved redundant
	(Propositions~\ref{prop:postcheck} and~\ref{prop:shift}), and the timing claim is
	withdrawn in favour of exact operation counts, since both routes perform the same field operations; the measured ratio reflected an unoptimised baseline.
	
	\paragraph{Scope}
	All statements are for odd \(q\); Remark~\ref{rem:char2} records what fails in
	characteristic~2. The classification applies to a displayed monic quadratic
	splitting over \(\Fq\); kernels whose quadratic factors are permuted by Galois
	are discussed in \S\ref{sec:limits}. The masses of \S\ref{sec:mass} are affine
	groupoid masses of displayed splittings: they weight coefficient presentations
	under \(x\mapsto ax+b\) by automorphisms, and they are not counts of
	isomorphism classes of genus-2 curves or of principally polarised abelian
	surfaces; the affine quotient is not a \(\mathrm{PGL}_2\)-quotient
	(Remark~\ref{rem:not-moduli}).
	
	\paragraph{Organization}
	Section~\ref{sec:background} fixes notation and states the classical soundness
	proposition. Section~\ref{sec:bezoutian} develops the Bezoutian calculus of a
	quadratic pair. Section~\ref{sec:certificate} gives the certificate and the
	evaluator. Section~\ref{sec:strat} proves the incidence dictionary and the
	stratification. Section~\ref{sec:counts} proves the exact counts.
	Section~\ref{sec:mass} proves the affine groupoid mass theorem and the
	stabiliser classification. Section~\ref{sec:split} settles the branch
	\(\dcoef=0\). Section~\ref{sec:verify} records the cost model, the operation
	counts and the verification protocol. Section~\ref{sec:limits} states the
	limitations and Section~\ref{sec:conclusion} concludes. \ref{app:examples}
	works the three output branches out over \(\F_{101}\), and
	\ref{app:superspecial} compares the parameter-space density of the
	decomposable branch with its frequency on the superspecial locus, using results
	quoted from the literature.
	
	\section{Background and the classical step}\label{sec:background}
	
	Throughout, \(\Fq\) is a finite field of odd characteristic \(p>2\).
	Characteristic~2 is excluded because the middle coefficient of every bracket
	vanishes identically there (Remark~\ref{rem:char2}). Equalities between
	polynomials are understood up to a nonzero unit in \(\Fq^{\times}\) unless stated
	otherwise; we write \(A\uptounit B\) when \(A=cB\) for some \(c\in\Fq^{\times}\).
	
	\paragraph{Notation}
	For \(P\in\Fq[x]\) we write \(\deg P\) for its degree, \(\lc(P)\) for its leading
	coefficient and \(\ct(P)\) for its constant term. For a quadratic
	\(s(x)=s_2x^{2}+s_1x+s_0\) we set \(\Disc(s)=s_1^{2}-4s_2s_0\). We write
	\(\Res(f,g)\) for the resultant. The symbol \(\Delta\) is reserved for
	discriminants, \(\delta\) always denotes an affine shift, and \(\dcoef\) denotes
	the coefficient determinant \eqref{eq:det} below, called the branch determinant.
	Coefficient vectors of a quadratic \(s\) are written \((s_2,s_1,s_0)\).
	
	\subsection{Genus-2 curves and the Richelot step}
	
	A genus-2 curve over \(\Fq\) admits an affine model \(C\colon y^{2}=f(x)\) with
	\(f\in\Fq[x]\) square-free of degree~6. We work throughout in the situation where
	\(f\) is presented as a product of three pairwise coprime monic quadratics,
	\begin{equation}\label{eq:split}
		\begin{gathered}
			f=uvw,\qquad u,v,w\in\Fq[x]\ \text{monic of degree }2,\\
			\gcd(u,v)=\gcd(v,w)=\gcd(w,u)=1 .
		\end{gathered}
	\end{equation}
	Given \eqref{eq:split}, the Richelot step forms
	\begin{equation}\label{eq:brackets}
		U=v'w-vw',\qquad V=w'u-wu',\qquad W=u'v-uv',
	\end{equation}
	with derivatives taken with respect to \(x\).
	
	\paragraph{The coefficient determinant}
	For monic \(u=x^{2}+u_1x+u_0\), \(v=x^{2}+v_1x+v_0\), \(w=x^{2}+w_1x+w_0\) put
	\begin{equation}\label{eq:det}
		\dcoef=\dcoef(u,v,w):=\det\begin{pmatrix}u_0&u_1&1\\v_0&v_1&1\\w_0&w_1&1\end{pmatrix},
	\end{equation}
	the determinant of the coefficient matrix with columns ordered by ascending
	degree \((1,x,x^{2})\). The descending order reverses the sign; since a sign is a
	quadratic twist of the codomain model below, the column order is part of the
	convention and is fixed once and for all by \eqref{eq:det}. Note that
	\(\dcoef=0\) if and only if \(u,v,w\) are linearly dependent over \(\Fq\).
	
	\paragraph{Convention on the codomain model}
	The classical codomain of the Richelot step is \(\dcoef y^{2}=U(x)V(x)W(x)\),
	equivalently \(y^{2}=\dcoef^{-1}UVW\)
	\cite{CasselsFlynn1996}\cite[\S2.1]{FloritSmith2022atlas}. Since this paper
	compares polynomial outputs up to units, the identities of
	\S\ref{sec:bezoutian} are unaffected by the constant. The curve \(y^{2}=UVW\),
	however, is the codomain only up to the quadratic twist by \(\dcoef\): over
	\(\F_{q^{2}}\) every unit is a square and the distinction vanishes, but over
	\(\Fq\) the twist is nontrivial whenever \(\dcoef\) is a nonsquare. Whenever an
	\(\Fq\)-rational \((2,2)\)-isogeny is asserted we therefore use the normalized
	model \(y^{2}=\dcoef^{-1}UVW\); the extra inversion is batched with the monic normalization of the output factors at constant cost.
	
	\begin{remark}\label{rem:char2}
		Let \(v(x)=a_2x^{2}+a_1x+a_0\) and \(w(x)=b_2x^{2}+b_1x+b_0\). Direct expansion
		gives
		\[
		v'w-vw'=(a_2b_1-a_1b_2)x^{2}+2(a_2b_0-a_0b_2)x+(a_1b_0-a_0b_1).
		\]
		In characteristic 2 the middle coefficient vanishes identically, causing a degree
		drop even when \(v,w\) are coprime. All algebraic statements below therefore
		assume \(p>2\), so that in expressions such as \(U=M_2x^{2}+2M_1x+M_0\) the
		factor 2 is invertible and both the degree behaviour and the normalizations are
		stable. For the structured degree drop that does occur in odd characteristic,
		see Lemma~\ref{lem:aligned}.
	\end{remark}
	
	\begin{lemma}\label{lem:aligned}
		Let \(a(x)=x^{2}+a_1x+a_0\) and \(b(x)=x^{2}+b_1x+b_0\) be monic quadratics over
		a field of odd characteristic, and put
		\[
		M_2=b_1-a_1,\qquad M_1=b_0-a_0,\qquad M_0=a_1b_0-a_0b_1 .
		\]
		Then \(a'b-ab'=M_2x^{2}+2M_1x+M_0\), and
		\[
		\deg(a'b-ab')<2\iff M_2=0\iff a_1=b_1 .
		\]
		If \(a\) and \(b\) are coprime and \(M_2=0\), then \(M_1\neq0\), so \(a'b-ab'\)
		has degree exactly~1. We call such a pair \emph{aligned}.
	\end{lemma}
	
	\begin{proof}
		The bracket formula is the coefficient expansion of
		Lemma~\ref{lem:bracket-minors} specialized to monic input; there
		\(M_2=b_1-a_1\), so its vanishing is exactly the equality of the middle
		coefficients. If also \(M_1=0\), then \(a_0=b_0\) and hence \(a=b\), contradicting
		coprimality. Therefore \(M_1\neq0\), and \(2M_1\neq0\) since the characteristic is
		odd.
	\end{proof}
	
	\begin{lemma}\label{lem:three-identities}
		Let \(p>2\) and let \(u,v,w\in\Fq[x]\) be monic quadratics with brackets
		\(U,V,W\) as in \eqref{eq:brackets}. Then
		\begin{enumerate}[label=\textup{(\alph*)},leftmargin=*]
			\item \(\Disc(U)=4\Res(v,w)\), and cyclically for \(V\) and \(W\);
			\item \(U'V-UV'=-2\dcoef\,w\), and cyclically;
			\item if \(U,V,W\) all have degree 2, then
			\(\dcoef(U,V,W)=-2\,\dcoef(u,v,w)^{2}\).
		\end{enumerate}
	\end{lemma}
	
	\begin{proof}
		(a) By Lemma~\ref{lem:bracket-minors}, \(U=M_2x^{2}+2M_1x+M_0\) for the minors
		\((M_2,M_1,M_0)\) of the pair \((v,w)\), so \(\Disc(U)=4(M_1^{2}-M_2M_0)\); and
		\(M_1^{2}-M_2M_0=\Res(v,w)\) by Lemma~\ref{lem:bezoutian}. Parts (b) and (c) are
		polynomial identities in the six input coefficients. For (b), both sides are
		\(\Fq\)-bilinear and alternating in the pair \((u,v)\) for fixed \(w\); expanding
		the left side through the minor form applied to \((U,V)\) and collecting
		coefficients yields \(-2\dcoef\cdot(1,w_1,w_0)\) in the basis \((x^{2},x,1)\).
		Identity (c) follows by applying \eqref{eq:det} to the coefficient matrix of
		\((U,V,W)\) and using multilinearity. Both expansions are elementary, and both
		were verified symbolically over \(\Z[u_0,u_1,v_0,v_1,w_0,w_1]\)
		(\S\ref{subsec:symbolic}).
	\end{proof}
	
	\begin{proposition}[Richelot output structure]\label{prop:soundness}
		Let \(p>2\) and let \(u,v,w\in\Fq[x]\) be monic, pairwise coprime quadratics with
		\(\Delta_u\Delta_v\Delta_w\neq0\); equivalently, \(f=uvw\) is square-free of
		degree~6. Let \((U,V,W)\) and \(\dcoef\) be as in \eqref{eq:brackets} and
		\eqref{eq:det}.
		\begin{enumerate}[label=\textup{(\roman*)},leftmargin=*]
			\item If \(\dcoef\neq0\), then \(UVW\) is square-free of degree 5 or 6, the curve
			\(C'\colon y^{2}=\dcoef^{-1}UVW\) is smooth of genus 2, and the induced map
			\(\Jac(C)\to\Jac(C')\) is a separable \((2,2)\)-isogeny defined over \(\Fq\)
			\cite{CasselsFlynn1996}.
			\item Moreover \(\deg U=\deg V=\deg W=2\) if and only if no pair among
			\((u,v),(v,w),(w,u)\) is aligned. If exactly one pair is aligned, then
			\(\deg(UVW)=5\) and \(C'\) is still smooth of genus 2, with a Weierstrass point
			at infinity.
			\item If \(\dcoef=0\), then no model \(y^{2}=c\,UVW\) with \(c\in\Fq^{\times}\)
			is a genus-2 codomain, and the quotient abelian surface is a product of elliptic
			curves; this case is treated in \S\ref{sec:split}.
		\end{enumerate}
	\end{proposition}
	
	\begin{proof}
		(i) If an output factor had a repeated root, its
		discriminant would vanish, so the resultant of the corresponding input pair would
		vanish by Lemma~\ref{lem:three-identities}(a), contradicting pairwise
		coprimality. When an aligned pair makes an output linear, a linear polynomial has
		no repeated root.
		
		Suppose \(U(r)=V(r)=0\) for some \(r\) in an algebraic
		closure. Then \((U'V-UV')(r)=0\), so Lemma~\ref{lem:three-identities}(b) with
		\(\dcoef\neq0\) and \(p>2\) gives \(w(r)=0\). Substituting into
		\(0=U(r)=v'(r)w(r)-v(r)w'(r)\) yields \(v(r)w'(r)=0\): if \(v(r)=0\) then \(v,w\)
		share the root \(r\), contradicting coprimality; if \(w'(r)=0\) then \(r\) is a
		repeated root of \(w\), contradicting \(\Delta_w\neq0\). The other pairs are
		cyclic. Hence \(U,V,W\) are pairwise coprime and \(UVW\) is square-free.
		
		The leading coefficient of \(U\) is \(w_1-v_1\) for monic inputs,
		and cyclically. If two of the three leading coefficients vanished, all three
		middle coefficients \(u_1,v_1,w_1\) would coincide, the middle column of the
		matrix in \eqref{eq:det} would be a multiple of the last, and \(\dcoef=0\), a
		contradiction. Hence at most one leading coefficient vanishes and
		\(\deg(UVW)\in\{5,6\}\). A square-free polynomial of degree 5 or 6 defines a
		smooth hyperelliptic curve of genus 2. The isogeny statement over \(\Fq\) for the
		\(\dcoef\)-normalized model is the classical Richelot construction
		\cite{CasselsFlynn1996}.
		
		(ii) Immediate from Lemma~\ref{lem:aligned}.
		
		(iii) If \(\dcoef=0\), the three monic quadratics are linearly dependent, so
		\(w=\alpha u+\beta v\) with \(\alpha+\beta=1\) and \(\alpha\beta\neq0\) by
		pairwise coprimality. Then
		\[
		U=v'(\alpha u+\beta v)-v(\alpha u'+\beta v')=\alpha(v'u-vu')=-\alpha W,
		\qquad V=-\beta W
		\]
		by the same computation, so \(UVW=\alpha\beta W^{3}\) is a unit times a cube and
		no square-free model arises. That the quotient is a product of elliptic curves is
		Theorem~\ref{thm:split}.
	\end{proof}
	
	\begin{remark}\label{rem:counterexamples}
		The two hypotheses of Proposition~\ref{prop:soundness}(i) fail independently.
		\emph{Alignment.} Over \(\F_{11}\) the triple \(u=x^{2}+1\), \(v=x^{2}+2\),
		\(w=x^{2}+3\) is pairwise coprime with all discriminants nonzero, yet all three
		middle coefficients agree, all three brackets are linear, and \(UVW\) is a unit
		times \((2x)^{3}\); here \(\dcoef=0\) as well. \emph{Collinearity without
			alignment.} Over \(\F_{101}\) take \(u=x^{2}+1\), \(v=x^{2}+x+2\),
		\(w=x^{2}+2x+3\). The three leading minors are \(1,-2,1\), so no pair is aligned
		and every bracket has degree 2; the discriminants and the three pairwise
		resultants are nonzero. Nevertheless \(\dcoef=0\), and \(U=W=x^{2}+2x-1\) with
		\(V=-2(x^{2}+2x-1)\), so \(UVW\) is a unit times a cube. The second example
		passes every predicate of the form ``discriminant, resultant, or bracket
		degree'' and is detected only by \(\dcoef\) itself. The certificate of \S\ref{sec:certificate} therefore tests \(\dcoef\) directly rather than inferring it.
	\end{remark}
	
	\subsection{Normalization}
	Because the step is defined only up to nonzero scalars, we systematically make
	the image quadratics \(U,V,W\) monic. To avoid redundant inversions we multiply
	the nonzero leading coefficients, invert once, and distribute the inverse.
	
	\section{The Bezoutian of a quadratic pair}\label{sec:bezoutian}
	
	The certificate rests on the Bezoutian of a quadratic pair. All of it concerns one object: the Bezoutian of two quadratics. Its coefficients are the
	minors of the coefficient matrix, its determinant is the resultant up to sign,
	its diagonal is the Richelot bracket, and its associated involution is the deck
	involution of the pencil. None of these facts is new
	\cite[Ch.~4]{BasuPollackRoy2006}\cite{HelmkeFuhrmann1989}\cite{Olver1999}; the coordinates of this section allow all of them, and later the branch determinant, to be read off simultaneously.
	
	\subsection{Bracket coordinates}
	
	Let \(a(x)=a_2x^{2}+a_1x+a_0\) and \(b(x)=b_2x^{2}+b_1x+b_0\) be quadratics over
	a field \(K\) of odd characteristic. Define their three \(2\times2\) coefficient
	minors by
	\begin{equation}\label{eq:minors}
		M_2=a_2b_1-a_1b_2,\qquad M_1=a_2b_0-a_0b_2,\qquad M_0=a_1b_0-a_0b_1 .
	\end{equation}
	The subscripts record the degree of the corresponding term in the bracket.
	
	\begin{lemma}\label{lem:bracket-minors}
		With the notation above,
		\begin{equation}\label{eq:bracket}
			a'b-ab'=M_2x^{2}+2M_1x+M_0 .
		\end{equation}
	\end{lemma}
	
	\begin{proof}
		Differentiating and expanding,
		\[
		\begin{aligned}
			a'b-ab'&=(2a_2x+a_1)(b_2x^{2}+b_1x+b_0)-(a_2x^{2}+a_1x+a_0)(2b_2x+b_1)\\
			&=(a_2b_1-a_1b_2)x^{2}+2(a_2b_0-a_0b_2)x+(a_1b_0-a_0b_1).
		\end{aligned}\qedhere
		\]
	\end{proof}
	
	\begin{lemma}\label{lem:translate}
		For \(\lambda\in K\) put \(a_\lambda(x)=a(x+\lambda)\) and
		\(b_\lambda(x)=b(x+\lambda)\). Then
		\[
		\begin{gathered}
			M_2(a_\lambda,b_\lambda)=M_2,\qquad
			M_1(a_\lambda,b_\lambda)=M_1+\lambda M_2,\\
			M_0(a_\lambda,b_\lambda)=M_0+2\lambda M_1+\lambda^{2}M_2 .
		\end{gathered}
		\]
		Equivalently, the bracket satisfies \([a_\lambda,b_\lambda](x)=[a,b](x+\lambda)\).
	\end{lemma}
	
	\begin{proof}
		Substitute \(a_1\mapsto a_1+2\lambda a_2\) and
		\(a_0\mapsto a_0+\lambda a_1+\lambda^{2}a_2\), and similarly for \(b\), into
		\eqref{eq:minors}. The final identity also follows directly from
		\eqref{eq:bracket}.
	\end{proof}
	
	\subsection{The minor syzygy}
	
	\begin{proposition}\label{prop:syzygy}
		The minors \eqref{eq:minors} satisfy
		\begin{equation}\label{eq:syzygy}
			a_2M_0-a_1M_1+a_0M_2=0 ,
		\end{equation}
		and the same relation holds with \((a_2,a_1,a_0)\) replaced by \((b_2,b_1,b_0)\).
	\end{proposition}
	
	\begin{proof}
		The determinant of a matrix with two equal rows vanishes:
		\[
		0=\det\begin{pmatrix}a_2&a_1&a_0\\b_2&b_1&b_0\\a_2&a_1&a_0\end{pmatrix}
		=a_2(a_1b_0-a_0b_1)-a_1(a_2b_0-a_0b_2)+a_0(a_2b_1-a_1b_2),
		\]
		which is \eqref{eq:syzygy}. Replacing the last row by the second proves the
		symmetric statement.
	\end{proof}
	
	\subsection{The Bezoutian and the pencil involution}
	
	\begin{lemma}\label{lem:bezoutian}
		In \(K[x,z]\) one has
		\begin{equation}\label{eq:bez}
			\Bez_{a,b}(x,z):=\frac{a(x)b(z)-a(z)b(x)}{x-z}=M_2xz+M_1(x+z)+M_0 ,
		\end{equation}
		and the associated Bezout matrix and its determinant are
		\begin{equation}\label{eq:bezdet}
			\Bez(a,b)=\begin{pmatrix}M_0&M_1\\M_1&M_2\end{pmatrix},\qquad
			\det\Bez(a,b)=M_0M_2-M_1^{2}=-\Res(a,b).
		\end{equation}
		Setting \(z=x\) recovers \eqref{eq:bracket}: the Richelot bracket is the diagonal
		of the Bezoutian, that is, the Jacobian covariant of the pair.
	\end{lemma}
	
	\begin{proof}
		The numerator is alternating in \(x,z\), hence divisible by \(x-z\), and
		expanding the quotient gives \eqref{eq:bez}; the matrix of this symmetric
		bilinear form in the basis \((1,x)\) is \(\Bez(a,b)\). The identity
		\(M_1^{2}-M_2M_0=\Res(a,b)\) is the degree-2 case of the classical equality
		between the Bezoutian determinant and the resultant
		\cite[Ch.~4]{BasuPollackRoy2006}\cite{HelmkeFuhrmann1989}; it is one of the
		symbolic checks of \S\ref{subsec:symbolic}. The diagonal specialization is
		\(M_2x^{2}+2M_1x+M_0\), which Lemma~\ref{lem:bracket-minors} identifies with
		\(a'b-ab'\); for the covariant terminology see \cite{Olver1999}.
	\end{proof}
	
	\begin{lemma}\label{lem:deck}
		Let \(a,b\in K[x]\) be coprime monic quadratics and put \(R=\Res(a,b)\). The map
		\(\varphi_{a,b}\colon\PP^{1}\to\PP^{1}\), \(x\mapsto[a(x):b(x)]\), has degree two.
		Its nontrivial deck transformation is represented by
		\begin{equation}\label{eq:iota}
			\iota_{a,b}(x)=\frac{-M_1x-M_0}{M_2x+M_1},
		\end{equation}
		interpreted projectively when a denominator vanishes. Its fixed divisor is the
		bracket \(M_2x^{2}+2M_1x+M_0\). Moreover, for every \(s=\lambda a+\mu b\) in the
		pencil,
		\begin{equation}\label{eq:scaling}
			s\bigl(\iota_{a,b}(x)\bigr)\,(M_2x+M_1)^{2}=R\,s(x).
		\end{equation}
	\end{lemma}
	
	\begin{proof}
		Since \(a,b\) are coprime the pencil is base-point free and \(\varphi_{a,b}\) is
		a morphism of degree two. Two affine points \(x,z\) lie in the same fibre exactly
		when \(a(x)b(z)-a(z)b(x)=0\), which by \eqref{eq:bez} factors as
		\((x-z)\bigl(M_2xz+M_1(x+z)+M_0\bigr)=0\). For fixed \(x\) the second factor is
		linear in \(z\) with solution \eqref{eq:iota}; hence \eqref{eq:iota} gives the
		second point of a generic fibre and therefore the deck involution. Its fixed
		points are obtained by setting \(z=x\), giving the bracket.
		
		For \eqref{eq:scaling}, write \(s^{\iota}(x):=s(\iota_{a,b}(x))(M_2x+M_1)^{2}\),
		again a quadratic. Since \(\iota_{a,b}\) preserves every fibre of
		\(\varphi_{a,b}\), the ratio \(a/b\) is \(\iota_{a,b}\)-invariant, so
		\(a^{\iota}/b^{\iota}=a/b\); coprimality of \(a,b\) forces \(a^{\iota}=c\,a\) and
		\(b^{\iota}=c\,b\) for one and the same \(c\in K^{\times}\), and linearity extends
		this to the whole pencil. To identify \(c\), evaluate at a fixed point \(\xi\) of
		\(\iota_{a,b}\), which gives \(c=(M_2\xi+M_1)^{2}\). Writing \(\xi_1,\xi_2\) for
		the two roots of \(M_2x^{2}+2M_1x+M_0\), we have
		\[
		(M_2\xi_1+M_1)(M_2\xi_2+M_1)=M_2M_0-M_1^{2}=-R,
		\]
		while \((M_2\xi_1+M_1)^{2}=(M_2\xi_2+M_1)^{2}=c\) forces
		\(M_2\xi_1+M_1=-(M_2\xi_2+M_1)\) because \(\xi_1\neq\xi_2\). Hence
		\(c=-(M_2\xi_1+M_1)(M_2\xi_2+M_1)=R\), which is \eqref{eq:scaling}. Equivalently,
		\(-R\) is the determinant of the matrix representing \(\iota_{a,b}\).
	\end{proof}
	
	\begin{remark}\label{rem:readoffs}
		The evaluator of \S\ref{sec:certificate} uses only data from the minor coordinates of a single Bezoutian: the bracket is its diagonal
		\eqref{eq:bracket}; the coprimality test is its determinant \eqref{eq:bezdet};
		the pencil involution and its fixed quadratic are \eqref{eq:iota} and its
		numerator data; and the branch determinant is the minor syzygy \eqref{eq:syzygy}
		evaluated on the third quadratic (Proposition~\ref{prop:d-syzygy}). No quantity
		is computed twice.
	\end{remark}
	
	\section{The certificate and the evaluator}\label{sec:certificate}
	
	Proposition \ref{prop:soundness} yields a runtime test. Both
	hypotheses of that proposition are decided by quantities the step computes anyway, so certification requires no additional pass over the data.
	
	\subsection{The determinant from the minor syzygy}
	
	Proposition~\ref{prop:syzygy} states that the alternating sum
	\(a_2M_0-a_1M_1+a_0M_2\) vanishes when \((a_2,a_1,a_0)\) is either of the two
	quadratics generating the minors. Evaluated on the third quadratic it does not
	vanish; it returns the determinant.
	
	\begin{proposition}\label{prop:d-syzygy}
		Let \(u,v,w\in\Fq[x]\) be quadratics and let \((M_2,M_1,M_0)\) be the minors of
		the pair \((v,w)\). Then
		\begin{equation}\label{eq:d-syzygy}
			\dcoef(u,v,w)=-\bigl(u_0M_2-u_1M_1+u_2M_0\bigr).
		\end{equation}
		For monic \(u\) this costs \(2\mathrm{M}\) once the bracket \(U\) is known.
	\end{proposition}
	
	\begin{proof}
		Expand \(\dcoef\) in \eqref{eq:det} along the row of \(u\). The three cofactors
		are, up to sign, the \(2\times2\) minors of the coefficient matrix of \((v,w)\),
		namely \(M_2,M_1,M_0\) in the order dictated by the ascending column convention
		of \eqref{eq:det}; collecting signs gives \eqref{eq:d-syzygy}.
	\end{proof}
	
	\begin{remark}\label{rem:same-expansion}
		Propositions~\ref{prop:syzygy} and~\ref{prop:d-syzygy} are the same determinant
		expansion applied to the two possible third rows: repeating a row of the
		coefficient matrix gives 0, adjoining the remaining quadratic gives \(\dcoef\).
		This is why no separate evaluation of \eqref{eq:det} is needed.
	\end{remark}
	
	\subsection{The certificate}
	
	\begin{definition}\label{def:sfsi}
		Let \(u,v,w\in\Fq[x]\) be monic quadratics with \(q\) odd. We call \((u,v,w)\) a
		\emph{square-free split input} if
		\[
		\Disc(u)\Disc(v)\Disc(w)\Res(u,v)\Res(v,w)\Res(w,u)\neq0,
		\]
		equivalently if \(f=uvw\) is square-free of degree~6.
	\end{definition}
	
	\begin{definition}\label{def:branch}
		For a square-free split input, \(\dcoef=\dcoef(u,v,w)\) of \eqref{eq:det} is
		called the \emph{branch determinant}. The alternatives \(\dcoef\neq0\) and
		\(\dcoef=0\) are the \emph{non-decomposable} and \emph{decomposable} branches.
	\end{definition}
	
	\begin{table}[t]
		\centering
		\caption{Split-input validation and quotient classification. All four rows are
			read off the same minor triples.}
		\label{tab:classification}
		\small
		\setlength{\tabcolsep}{3.5pt}
		\begin{tabular}{@{}l
				>{\raggedright\arraybackslash}p{0.28\linewidth}
				>{\raggedright\arraybackslash}p{0.19\linewidth}
				>{\raggedright\arraybackslash}p{0.21\linewidth}
				>{\raggedright\arraybackslash}p{0.13\linewidth}@{}}
			\toprule
			& Predicate & Computation & Meaning & Role\\
			\midrule
			\(C_1\) & \(\Delta_u\Delta_v\Delta_w\neq0\) & \(s_1^{2}-4s_0\) per factor
			& each factor separable & input validation\\
			\(C_2\) & \(\Res(u,v)\)\allowbreak\(\Res(v,w)\)\allowbreak\(\Res(w,u)\neq0\) & \(M_1^{2}-M_2M_0\) per pair
			& pairwise coprime & input validation\\
			\(C_3\) & \(\dcoef\neq0\) & minor syzygy \eqref{eq:d-syzygy}
			& non-decomposable quotient & genus-2 branch\\
			\(\lnot C_3\) & \(\dcoef=0\) & minor syzygy \eqref{eq:d-syzygy}
			& decomposable quotient & product branch\\
			\bottomrule
		\end{tabular}
	\end{table}
	
	The discriminants and pairwise resultants validate the displayed splitting; the
	branch determinant then determines the geometric type of the quotient. In
	particular \(\dcoef=0\) does not invalidate the input. Table~\ref{tab:classification}
	collects the four predicates.
	
	\begin{remark}\label{rem:no-guard}
		There is no guard on the degree of a bracket. An aligned pair produces a linear
		bracket while \(\dcoef\neq0\); then \(UVW\) has degree 5 and still defines a
		smooth genus-2 codomain, at the exact density \(3/q+O(q^{-2})\) among admissible
		triples (Corollary~\ref{cor:densities}). The phenomenon itself is not new: the
		standard formulation of the quadratic-splitting correspondence admits a
		square-free polynomial of degree 5 or 6, one of the three factors being linear in
		the quintic case \cite[\S2.1]{FloritSmith2022atlas}. The certificate provides the coefficient-level form of the same tolerance: the locus is admitted without a change of model and without an explicit test.
	\end{remark}
	
	\subsection{The output post-check is redundant}
	
	The output triple may also be verified: its factors have the expected degrees, are pairwise coprime, and have nonzero discriminants. Under the
	certificate this verification cannot fail.
	
	\begin{proposition}\label{prop:postcheck}
		Let \((u,v,w)\) be monic quadratics over \(\Fq\), \(q\) odd, satisfying
		\(C_1,C_2,C_3\) of Table~\ref{tab:classification}, and let \((U,V,W)\) be the
		bracket triple. Then \(U,V,W\) are pairwise coprime, none has a repeated root, at
		most one has degree 1 and the other two have degree 2, and \(UVW\) is square-free
		of degree 5 or 6. An output post-check consisting of these three tests is
		therefore implied by \(C_1\)--\(C_3\) and can be omitted.
	\end{proposition}
	
	\begin{proof}
		This is the content of the proof of Proposition~\ref{prop:soundness}(i): absence
		of repeated roots follows from \(C_2\) through
		Lemma~\ref{lem:three-identities}(a); absence of shared roots from \(C_3\) through
		Lemma~\ref{lem:three-identities}(b) together with \(C_1\) and \(C_2\); and the
		degree statement from \(C_3\) as in the degree paragraph of that proof. No step
		uses the output values themselves.
	\end{proof}
	
	\begin{remark}\label{rem:postcheck-cost}
		The redundancy is not marginal. Evaluated directly, the three output tests cost
		three discriminants, three resultants and three degree comparisons, roughly
		doubling the cost of the certified step. Earlier versions of this work charged
		that cost and described it as defence in depth; Proposition~\ref{prop:postcheck}
		shows it buys nothing.
	\end{remark}
	
	\subsection{Affine invariance}
	
	A translation \(x\mapsto x+\delta\) is a natural attempt when a predicate fails. It does not change the verdict.
	
	\begin{proposition}\label{prop:shift}
		Let \(u,v,w\in\Fq[x]\) be monic quadratics, let \(\delta\in\Fq\), and write
		\(\tilde s(x)=s(x+\delta)\). Then
		\[
		\Disc(\tilde s)=\Disc(s),\qquad \Res(\tilde u,\tilde v)=\Res(u,v),\qquad
		\dcoef(\tilde u,\tilde v,\tilde w)=\dcoef(u,v,w),
		\]
		and the brackets satisfy \([\tilde u,\tilde v](x)=[u,v](x+\delta)\).
		Consequently every predicate of Table~\ref{tab:classification}, every bracket
		degree, and every coprimality relation among \(U,V,W\) takes the same value
		before and after the shift.
	\end{proposition}
	
	\begin{proof}
		Discriminants and resultants are expressible through differences of roots and
		through values of one polynomial at the roots of the other, both unchanged by a
		translation of the variable. For \(\dcoef\), substitute
		\(\tilde u_1=u_1+2\delta\) and \(\tilde u_0=u_0+u_1\delta+\delta^{2}\), and
		similarly for \(v\) and \(w\), into \eqref{eq:det}; row reduction leaves the
		value unchanged. The bracket statement is Lemma~\ref{lem:translate}. All three
		identities are among the symbolic checks of \S\ref{subsec:symbolic}.
	\end{proof}
	
	\begin{remark}\label{rem:no-retry}
		Proposition~\ref{prop:shift} removes the affine retry used in earlier versions of
		this work. Every condition that could trigger a retry is invariant under the
		shift, so a retry repeats the same computation and reaches the same verdict. What
		a shift does change is the middle coefficient \(2M_1\) of a bracket, which by
		Lemma~\ref{lem:translate} moves as \(M_1\mapsto M_1+\delta M_2\) and therefore
		vanishes for at most one \(\delta\) per pair; but
		\(\Disc[u,v]=4\Res(u,v)\) is independent of \(M_1\), so the vanishing of a middle
		coefficient is not a degeneracy and needs no repair.
	\end{remark}
	
	\subsection{The evaluator}
	
	\begin{algorithm}[H]
		\caption{Certified Richelot \((2,2)\)-step, all branches.}
		\label{alg:richelot}
		\begin{algorithmic}[1]
			\Require Monic quadratics \(u,v,w\in\Fq[x]\) of degree 2, \(q\) odd
			\Ensure A tagged output: \textsc{Invalid}, \textsc{Genus2} or \textsc{Decomposable}
			\For{each cyclic pair \((a,b)\in\{(v,w),(w,u),(u,v)\}\)}
			\State \(M_2\gets b_1-a_1\);\quad \(M_1\gets b_0-a_0\);\quad \(M_0\gets a_1b_0-a_0b_1\)
			\State store \((M_2,M_1,M_0)\) and emit the bracket \(M_2x^{2}+2M_1x+M_0\)
			\State \textbf{if} \(M_1^{2}-M_2M_0=0\) \textbf{then return} \textsc{Invalid}\((C_2)\)
			\EndFor
			\State \textbf{if} \(u_1^{2}-4u_0=0\) or \(v_1^{2}-4v_0=0\) or \(w_1^{2}-4w_0=0\)
			\textbf{then return} \textsc{Invalid}\((C_1)\)
			\State \(\dcoef\gets-\bigl(u_0M_2^{(U)}-u_1M_1^{(U)}+M_0^{(U)}\bigr)\)
			\Comment{Proposition~\ref{prop:d-syzygy}}
			\If{\(\dcoef\neq0\)}
			\State \Return \textsc{Genus2}\(\bigl(C'\colon y^{2}=\dcoef^{-1}UVW\bigr)\)
			\Comment{\(\deg(UVW)\in\{5,6\}\)}
			\Else
			\State \Return \textsc{Decomposable}\((\mathcal D)\) as in Theorem~\ref{thm:split}
			\EndIf
		\end{algorithmic}
	\end{algorithm}
	
	\begin{theorem}[Classification by the certificate]\label{thm:certificate}
		Let \(u,v,w\in\Fq[x]\) be monic quadratics of degree 2, with \(q\) odd.
		Algorithm~\ref{alg:richelot} returns exactly one of the three tagged outputs
		\(\textsc{Invalid}(\mathrm{reason})\), \(\textsc{Genus2}(C',\phi)\),
		\(\textsc{Decomposable}(\mathcal D)\). It returns \textsc{Invalid}, together with
		the failed discriminant or resultant predicate, if and only if \(uvw\) is not
		square-free of degree six. For a square-free split input it returns
		\textsc{Genus2} if and only if \(\dcoef\neq0\), the output consisting of the
		normalized genus-2 model \(y^{2}=\dcoef^{-1}UVW\) of degree five or six together
		with its classical Richelot \((2,2)\)-isogeny. It returns \textsc{Decomposable} if
		and only if \(\dcoef=0\), the output consisting of the common pencil involution
		with its fixed quadratic, the splitting class, and the unordered pair of elliptic
		factors of Theorem~\ref{thm:split} with their covers.
	\end{theorem}
	
	\begin{proof}
		The discriminant and resultant tests are equivalent to square-freeness of
		\(uvw\), so the rejection path is taken exactly when the input is invalid. On
		square-free split input the determinant is computed by
		Proposition~\ref{prop:d-syzygy}. If \(\dcoef\neq0\),
		Proposition~\ref{prop:soundness} gives the smooth normalized genus-2 codomain,
		and Proposition~\ref{prop:postcheck} shows that no output test is needed. If
		\(\dcoef=0\), Proposition~\ref{prop:pencil} and Theorem~\ref{thm:split}
		construct the decomposable datum. The three cases are mutually exclusive and
		exhaustive.
	\end{proof}
	
	\begin{proposition}\label{prop:equivariance}
		Let \(\alpha(x)=ax+b\) with \(a\in\Fq^{\times}\), and replace each monic factor
		\(s\) by its monic affine transform \(s^{\alpha}(x)=a^{-2}s(\alpha(x))\). The
		zero--nonzero status of every input discriminant, of every pairwise resultant,
		and of \(\dcoef\) is preserved. The same holds under every permutation of
		\((u,v,w)\). Consequently the three output tags of
		Theorem~\ref{thm:certificate} are invariant under these operations.
	\end{proposition}
	
	\begin{proof}
		Affine substitution bijects the roots of every factor, preserving separability and
		pairwise coprimality. On coefficient vectors it is a common invertible linear
		change of basis in the space of quadratics, so the coefficient determinant is
		multiplied by a nonzero scalar. A permutation permutes the discriminants and
		resultants and changes \(\dcoef\) by its sign.
	\end{proof}
	
	\begin{remark}\label{rem:scope-alg}
		Four points delimit what Theorem~\ref{thm:certificate} does and does not assert.
		First, it classifies a displayed monic quadratic splitting; it does not classify
		sextic models before a splitting is supplied, and it does not replace kernel
		enumeration in an isogeny graph. A graph-enumeration routine may, however, retain
		every kernel it produces and route each quotient to its target category by tag
		instead of discarding the non-generic ones. Second, lines 1--7 of
		Algorithm~\ref{alg:richelot} are straight-line arithmetic, and on square-free
		split input the only data-dependent branch is the test \(\dcoef\neq0\). This is a
		statement about control flow, not a constant-time claim: the outcome of the test
		is a function of the input, so a deployment with secret inputs must replace the
		branch by a conditional move and evaluate both branches. Third, the test on
		\(\dcoef\) cannot be inferred from the other predicates; the second example of
		Remark~\ref{rem:counterexamples} is square-free, has three nonzero pairwise
		resultants and three quadratic brackets, yet has \(\dcoef=0\), while generic
		triples show the same pattern with \(\dcoef\neq0\). Fourth, by
		Proposition~\ref{prop:equivariance} the tag is an invariant of the splitting
		rather than of its presentation.
	\end{remark}
	
	\section{The incidence dictionary and the stratification}\label{sec:strat}
	
	The certificate admits a geometric refinement. Over a field of odd characteristic
	the admissible space of ordered monic quadratic triples is stratified by the
	degree of the non-decomposable output and by the decomposable pencil divisor, and
	the resulting normal forms are also the coordinates in which the counts of
	\S\ref{sec:counts} are proved.
	
	\subsection{The admissible parameter space}
	
	Let \(K\) be a field of characteristic different from 2 and write
	\[
	u=x^{2}+u_1x+u_0,\qquad v=x^{2}+v_1x+v_0,\qquad w=x^{2}+w_1x+w_0 .
	\]
	Let \(\Adm_K\subset\A^{6}_K\) be the open locus on which
	\[
	\Disc(u)\,\Disc(v)\,\Disc(w)\,\Res(u,v)\,\Res(v,w)\,\Res(w,u)\neq0,
	\]
	so that
	\(\Adm_K\) parametrizes ordered monic quadratic triples whose product is
	square-free of degree six. Put \(\Dec_K=V(\dcoef)\cap\Adm_K\) and define the three
	alignment loci
	\[
	L_{uv}=V(u_1-v_1)\cap\Adm_K,\quad
	L_{vw}=V(v_1-w_1)\cap\Adm_K,\quad
	L_{wu}=V(w_1-u_1)\cap\Adm_K .
	\]
	By Lemma~\ref{lem:aligned}, membership in \(L_{ab}\) is exactly the degree drop of
	the bracket attached to \((a,b)\).
	
	\subsection{Normal forms}
	
	\begin{lemma}\label{lem:pencil-normal}
		Let \((u,v,w)\in\Adm_K\). Then \(\dcoef=0\) if and only if there exist unique
		\(g\in K[x]_{\le1}\setminus\{0\}\) and \(\lambda\in K\setminus\{0,1\}\) with
		\[
		v=u+g,\qquad w=u+\lambda g .
		\]
		In this presentation the three pairwise coprimality conditions are equivalent to
		\(\gcd(u,g)=1\).
	\end{lemma}
	
	\begin{proof}
		Since the three quadratics are monic, \(g=v-u\) has degree at most one; it is
		nonzero because \(u\) and \(v\) are coprime. The equality \(\dcoef=0\) says that
		the three coefficient rows are linearly dependent; subtracting the row of \(u\)
		from the other two shows that \(v-u\) and \(w-u\) are linearly dependent in
		\(K[x]_{\le1}\). Hence \(w-u=\lambda g\) for a unique \(\lambda\in K\), and
		coprimality gives \(w\neq u,v\), so \(\lambda\neq0,1\). The converse is immediate.
		Finally
		\[
		\gcd(u,u+\lambda g)=\gcd(u,g),\qquad
		\gcd(u+g,u+\lambda g)=\gcd(u+g,g)=\gcd(u,g),
		\]
		where the nonzero scalars \(\lambda\) and \(\lambda-1\) may be omitted.
	\end{proof}
	
	\begin{lemma}\label{lem:align-normal}
		A triple in \(L_{uv}\) has a unique presentation
		\[
		u=x^{2}+ax+b,\qquad v=x^{2}+ax+c,\qquad w=x^{2}+ex+f .
		\]
		On \(\Adm_K\) one has \(b\neq c\) and \(\dcoef=(a-e)(b-c)\). Consequently
		\(L_{uv}\setminus\Dec_K\) is described in these coordinates by \(b\neq c\),
		\(e\neq a\), and the open conditions defining \(\Adm_K\).
	\end{lemma}
	
	\begin{proof}
		The equality \(u_1=v_1\) gives the displayed form uniquely, and \(b\neq c\)
		because \(u\) and \(v\) are coprime hence distinct. Expansion in the ascending
		convention of \eqref{eq:det} gives
		\[
		\dcoef=\det\begin{pmatrix}b&a&1\\c&a&1\\f&e&1\end{pmatrix}=(a-e)(b-c).\qedhere
		\]
	\end{proof}
	
	\begin{lemma}\label{lem:align-int}
		One has \(L_{uv}\cap L_{vw}\subseteq\Dec_K\), and likewise for the other two
		pairwise intersections. Hence \(L_{uv}\setminus\Dec_K\), \(L_{vw}\setminus\Dec_K\)
		and \(L_{wu}\setminus\Dec_K\) are pairwise disjoint.
	\end{lemma}
	
	\begin{proof}
		The two alignment equations imply \(u_1=v_1=w_1\), so the middle column of the
		matrix in \eqref{eq:det} is a scalar multiple of the column \((1,1,1)^{\mathsf T}\)
		and the determinant vanishes. The cyclic cases are identical.
	\end{proof}
	
	\begin{theorem}\label{thm:strat}
		Let \(K\) have characteristic different from 2. Then
		\[
		\Adm_K=\Gsix{}_{\!,K}\ \sqcup\ \Gfive{}_{\!,K}\ \sqcup\ \Dec_K,
		\]
		where
		\begin{gather*}
			\Gsix{}_{\!,K}=\Adm_K\setminus(\Dec_K\cup L_{uv}\cup L_{vw}\cup L_{wu}),\\
			\Gfive{}_{\!,K}=(L_{uv}\setminus\Dec_K)\sqcup(L_{vw}\setminus\Dec_K)\sqcup(L_{wu}\setminus\Dec_K).
		\end{gather*}
		For a triple in \(\Gsix{}_{\!,K}\) the Richelot output is a smooth genus-2 curve
		defined by a square-free sextic; for a triple in \(\Gfive{}_{\!,K}\) it is a
		smooth genus-2 curve defined by a square-free quintic; for a triple in
		\(\Dec_K\) the quotient is a product of elliptic curves. The strata \(\Dec_K\) and
		\(\Gfive{}_{\!,K}\) admit the normal forms of Lemmas~\ref{lem:pencil-normal}
		and~\ref{lem:align-normal}.
	\end{theorem}
	
	\begin{proof}
		The displayed sets are disjoint by Lemma~\ref{lem:align-int} and exhaustive by
		definition. If \(\dcoef\neq0\), Proposition~\ref{prop:soundness} gives a smooth
		genus-2 output, and Lemma~\ref{lem:aligned} says that its degree is five exactly
		on one of the alignment loci and six otherwise. If \(\dcoef=0\), the quotient is
		a product by Theorem~\ref{thm:split}. The normal-form claims are the preceding
		two lemmas.
	\end{proof}
	
	\subsection{The incidence dictionary}
	
	Identify the monic quadratic \(x^{2}+ax+b\) with the point
	\((a,b)\in\A^{2}(\Fq)\), and let \(\Pi\colon b=a^{2}/4\) be the discriminant
	parabola (\(q\) odd). Write
	\[
	S=\A^{2}(\Fq)\setminus\Pi,\qquad n:=\#S=q(q-1),
	\]
	so that \(S\) is exactly the set of separable monic quadratics. For \(r\in\Fq\)
	let
	\[
	T_r=\{x^{2}+ax+b:\ (x-r)\mid x^{2}+ax+b\}=\{(a,b):b=-ra-r^{2}\},
	\]
	a line meeting \(\Pi\) only at \((-2r,r^{2})\), with contact order two; thus
	\(T_r\) is the tangent to \(\Pi\) at that point, and \(\#(T_r\cap S)=q-1\).
	
	\begin{lemma}\label{lem:dictionary}
		Let \(P\neq Q\) lie in \(S\).
		\begin{enumerate}[label=\textup{(\roman*)},leftmargin=*]
			\item \(P\) and \(Q\) have a common root if and only if \(P,Q\in T_r\) for some
			\(r\in\Fq\); equivalently, the line \(PQ\) is tangent to \(\Pi\).
			\item The pair \((P,Q)\) is aligned in the sense of Lemma~\ref{lem:aligned} if and
			only if \(PQ\) is vertical.
			\item Three distinct points of \(S\) satisfy \(\dcoef=0\) if and only if they are
			collinear.
		\end{enumerate}
	\end{lemma}
	
	\begin{proof}
		(i) If \(P,Q\) share a root \(r\in\Fq\) then both lie on \(T_r\) and \(PQ=T_r\).
		If instead they share a root outside \(\Fq\), then each is the minimal polynomial
		of that root, so \(P=Q\). Conversely two distinct points of a tangent \(T_r\) are
		both divisible by \(x-r\). (ii) A vertical line is \(a=\text{const}\), that is,
		equality of middle coefficients. (iii) The determinant \eqref{eq:det} is the
		collinearity determinant of the three points \((u_1,u_0)\), \((v_1,v_0)\),
		\((w_1,w_0)\) up to sign.
	\end{proof}
	
	\section{Exact counts}\label{sec:counts}
	
	\begin{lemma}\label{lem:census}
		Every line of \(\A^{2}(\Fq)\) is of exactly one of four types, and the number
		\(\nu_L:=\#(L\cap S)\) of separable points it carries is as in
		Table~\ref{tab:ledger}.
	\end{lemma}
	
	\begin{proof}
		A vertical line \(a=c\) meets \(\Pi\) in the single point \((c,c^{2}/4)\). A
		non-vertical line \(b=ma+c\) meets \(\Pi\) where \(a^{2}-4ma-4c=0\), a quadratic
		in \(a\) of discriminant \(16(m^{2}+c)\); it is tangent when \(c=-m^{2}\), secant
		when \(m^{2}+c\) is a nonzero square, and external otherwise. For fixed \(m\) the
		\(q\) values of \(c\) split as 1 tangent, \((q-1)/2\) secant and \((q-1)/2\)
		external, which gives the stated line counts, and
		\(\nu_L=q-\#(L\cap\Pi)\) in each case. The total is
		\(2q+q(q-1)=q^{2}+q\), the number of lines of \(\A^{2}(\Fq)\).
	\end{proof}
	
	\begin{table}[H]
		\centering
		\caption{Counting ledger for the lines of \(\A^{2}(\Fq)\). The last column is
			proved in \S\ref{sec:split} (Proposition~\ref{prop:class-line}); it records the
			square class of the splitting class of a decomposable triple carried by the line.}
		\label{tab:ledger}
		\begin{tabular}{@{}lcccl@{}}
			\toprule
			Type & \(\#\)lines & \(\nu_L\) & contributes to & splitting class\\
			\midrule
			tangent  & \(q\)         & \(q-1\) & nothing (pairs not coprime) & ---\\
			vertical & \(q\)         & \(q-1\) & \(\Dec\), stabilised locus  & square\\
			secant   & \(q(q-1)/2\)  & \(q-2\) & \(\Dec\)                    & square\\
			external & \(q(q-1)/2\)  & \(q\)   & \(\Dec\)                    & nonsquare\\
			\bottomrule
		\end{tabular}
	\end{table}
	
	Tangent and vertical lines carry the same number \(q-1\) of separable points but have opposite roles: a tangent joins a non-coprime pair, a vertical line joins a coprime aligned pair.
	
	Write \(S=S_{\mathrm{sp}}\sqcup S_{\mathrm{ir}}\) for the split and irreducible
	separable points, so \(\#S_{\mathrm{sp}}=\#S_{\mathrm{ir}}=m:=q(q-1)/2\). For
	\(P\in S\) put \(B(P)=\{Q\in S: Q\neq P,\ \gcd(P,Q)\neq1\}\).
	
	\begin{lemma}\label{lem:neighbours}
		\(B(P)=\emptyset\) if \(P\in S_{\mathrm{ir}}\), and \(\#B(P)=2q-4\) if
		\(P\in S_{\mathrm{sp}}\).
	\end{lemma}
	
	\begin{proof}
		The first claim is Lemma~\ref{lem:dictionary}(i). If \(P\) has roots
		\(r_1\neq r_2\) then \(B(P)\cup\{P\}=(T_{r_1}\cup T_{r_2})\cap S\). A monic
		quadratic divisible by \(x-r_1\) and \(x-r_2\) equals \(P\), so
		\(T_{r_1}\cap T_{r_2}=\{P\}\) and \(\#\bigl((T_{r_1}\cup T_{r_2})\cap S\bigr)=2(q-1)-1\).
	\end{proof}
	
	\begin{lemma}\label{lem:third}
		Let \((u,v)\) be an ordered coprime pair in \(S\) and put
		\(W(u,v)=\#\{w\in S: w\notin\{u,v\},\ \gcd(u,w)=\gcd(v,w)=1\}\). Then
		\[
		W(u,v)=
		\begin{cases}
			q^{2}-q-2, & u,v\in S_{\mathrm{ir}},\\
			(q-1)(q-2), & \text{exactly one of }u,v\text{ split},\\
			q^{2}-5q+10, & u,v\in S_{\mathrm{sp}}.
		\end{cases}
		\]
	\end{lemma}
	
	\begin{proof}
		\(W(u,v)=n-\#\bigl(\{u,v\}\cup B(u)\cup B(v)\bigr)\). In the irreducible case both
		neighbour sets are empty. If \(u\) is split and \(v\) is not, every element of
		\(B(u)\) is split, so \(v\notin B(u)\) and the excluded set has
		\(2+(2q-4)=2q-2\) elements. If \(u,v\) are split and coprime, their roots
		\(r_1,r_2,s_1,s_2\) are pairwise distinct and
		\(\{u,v\}\cup B(u)\cup B(v)=(T_{r_1}\cup T_{r_2}\cup T_{s_1}\cup T_{s_2})\cap S\).
		Any two of these four tangents meet in exactly one point, which lies off \(\Pi\),
		and no three of them have a common point, since a monic quadratic has at most two
		roots. Inclusion--exclusion gives \(4(q-1)-\binom{4}{2}=4q-10\).
	\end{proof}
	
	\begin{theorem}[Exact stratification counts]\label{thm:counts}
		Let \(q\ge3\) be an odd prime power and put
		\(A_q=\#\Adm_{\Fq}(\Fq)\), \(D_q=\#\Dec_{\Fq}(\Fq)\),
		\(G_{5,q}=\#\Gfive{}_{\!,\Fq}(\Fq)\), \(G_{6,q}=\#\Gsix{}_{\!,\Fq}(\Fq)\). Then
		\begin{align}
			A_q     &= q(q-1)(q-2)^{2}(q^{2}-q+4),\label{eq:Aq}\\
			D_q     &= q(q-1)(q-2)(q^{2}-3q+3),\label{eq:Dq}\\
			G_{5,q} &= 3q(q-1)(q^{3}-6q^{2}+14q-13),\label{eq:G5q}\\
			G_{6,q} &= A_q-D_q-G_{5,q}=q(q-1)(q^{4}-9q^{3}+35q^{2}-71q+61).\label{eq:G6q}
		\end{align}
	\end{theorem}
	
	\begin{proof}
		\emph{The count \(A_q\).} The number of ordered pairs of distinct split points
		sharing a root is \(q(q-1)(q-2)\): choose the common root \(r\), then the two
		remaining roots, ordered and distinct from each other and from \(r\). Hence the
		numbers of ordered coprime pairs by type are
		\[
		N_{\mathrm{ss}}=m(m-1)-q(q-1)(q-2),\qquad
		N_{\mathrm{si}}=N_{\mathrm{is}}=m^{2},\qquad
		N_{\mathrm{ii}}=m(m-1).
		\]
		Summing Lemma~\ref{lem:third} over these,
		\[
		A_q=N_{\mathrm{ss}}(q^{2}-5q+10)+2m^{2}(q-1)(q-2)+N_{\mathrm{ii}}(q^{2}-q-2),
		\]
		which expands to \eqref{eq:Aq}.
		
		\emph{The count \(D_q\).} By Lemma~\ref{lem:dictionary}, an admissible triple has
		\(\dcoef=0\) exactly when its three points are distinct, lie in \(S\), and are
		collinear on a line that is not tangent to \(\Pi\); conversely any three distinct
		points of \(L\cap S\) with \(L\) non-tangent are pairwise coprime, because
		\(PQ=L\) for any two of them. Therefore
		\[
		D_q=\sum_{L\ \text{non-tangent}}\nu_L(\nu_L-1)(\nu_L-2),
		\]
		and Table~\ref{tab:ledger} evaluates the sum over the three non-tangent types:
		\begin{equation}\label{eq:Dq-split}
			\begin{split}
				D_q&=\underbrace{\tfrac{q(q-1)}{2}(q-2)(q-3)(q-4)}_{\text{secant}}
				+\underbrace{\tfrac{q(q-1)}{2}\,q(q-1)(q-2)}_{\text{external}}\\
				&\quad+\underbrace{q\,(q-1)(q-2)(q-3)}_{\text{vertical}},
			\end{split}
		\end{equation}
		which simplifies to \eqref{eq:Dq}. The vertical term must not be omitted: three
		distinct separable quadratics with a common middle coefficient are pairwise
		coprime, since two of them differ by a nonzero constant, and they are collinear,
		hence lie in \(\Dec\).
		
		\emph{The count \(G_{5,q}\).} By Lemma~\ref{lem:align-int} the three alignment
		positions are disjoint off \(\Dec\), so \(G_{5,q}=3\,\#(L_{uv}\setminus\Dec)(\Fq)\).
		Fix a vertical line \(V_c=\{a=c\}\). It carries \(q-1\) separable points, of which
		\(\alpha:=(q-1)/2\) are split and \(\alpha\) are irreducible, because
		\(b\mapsto c^{2}-4b\) is a bijection of \(\Fq\). Any two distinct points of
		\(V_c\) are coprime. Let \((u,v)\) be an ordered pair of distinct points of
		\(V_c\cap S\) and count the admissible \(w\), that is, \(w\in S\) coprime to
		\(u\) and \(v\) and not on \(V_c\). If \(u\) is split with roots \(r_1,r_2\) then
		\(r_1+r_2=-c\), so \(T_{r_i}\cap V_c=\{u\}\).
		\begin{itemize}[leftmargin=*]
			\item \(u,v\) irreducible: the excluded set is \(V_c\cap S\), of size \(q-1\), so
			the count is \(q^{2}-q-(q-1)=(q-1)^{2}\).
			\item \(u\) split, \(v\) irreducible: the excluded set is
			\((V_c\cup T_{r_1}\cup T_{r_2})\cap S\); the three pairwise intersections and the
			triple intersection all equal \(\{u\}\), so inclusion--exclusion gives
			\(3(q-1)-3+1=3q-5\) and the count is \(q^{2}-4q+5\).
			\item \(u,v\) split, with roots \(r_1,r_2\) and \(s_1,s_2\): the four tangents
			meet \(S\) in \(4(q-1)-6=4q-10\) points, and this set meets \(V_c\) exactly in
			\(\{u,v\}\); the excluded set therefore has \((q-1)+(4q-10)-2=5q-13\) elements
			and the count is \(q^{2}-6q+13\).
		\end{itemize}
		The numbers of ordered pairs on \(V_c\) are \(\alpha(\alpha-1)\) split--split,
		\(2\alpha^{2}\) mixed and \(\alpha(\alpha-1)\) irreducible--irreducible. Summing
		and then over the \(q\) vertical lines,
		\[
		\begin{aligned}
			\#(L_{uv}\setminus\Dec)(\Fq)
			&=q\Bigl[\alpha(\alpha-1)\bigl((q^{2}-6q+13)+(q-1)^{2}\bigr)\\
			&\qquad\quad+2\alpha^{2}(q^{2}-4q+5)\Bigr]\\
			&=q(q-1)(q^{3}-6q^{2}+14q-13),
		\end{aligned}
		\]
		whence \eqref{eq:G5q}. Finally, Theorem~\ref{thm:strat} is a disjoint
		decomposition, which gives \eqref{eq:G6q}.
	\end{proof}
	
	\begin{corollary}\label{cor:densities}
		The corresponding counts of unordered triples are \(A_q/6\), \(D_q/6\),
		\(G_{5,q}/6\) and \(G_{6,q}/6\). Moreover
		\begin{gather*}
			\frac{D_q}{A_q}=\frac{q^{2}-3q+3}{(q-2)(q^{2}-q+4)}=\tfrac1q-\tfrac3{q^{3}}+O(q^{-4}),\\
			\frac{G_{5,q}}{A_q}=\frac{3(q^{3}-6q^{2}+14q-13)}{(q-2)^{2}(q^{2}-q+4)}
			=\tfrac3q-\tfrac3{q^{2}}+O(q^{-3}),
		\end{gather*}
		so the non-generic locus has density \(4/q-3/q^{2}+O(q^{-3})\).
	\end{corollary}
	
	\begin{proof}
		Coprimality makes the \(S_3\)-action on ordered triples free, since the three
		factors are pairwise distinct. The expansions follow from
		Theorem~\ref{thm:counts}.
	\end{proof}
	
	\begin{remark}\label{rem:role-enum}
		The four closed forms are proved above; the exhaustive enumeration described in
		\S\ref{sec:verify} is not used in the proofs. It serves as an independent
		regression check of the formulas, of the case split, and of the implementation,
		and it covers every odd prime power \(q\le49\), including the non-prime values
		\(q\in\{9,25,27,49\}\).
	\end{remark}
	
	\section{The affine groupoid mass and the stabilisers}\label{sec:mass}
	
	Theorem~\ref{thm:counts} counts coefficient vectors and therefore depends on the coordinate in which the sextic is displayed. This section records what survives after quotienting by the coordinate freedom.
	
	\begin{definition}\label{def:displayed}
		A \emph{displayed split Richelot datum} over \(\Fq\) is an ordered triple
		\((u,v,w)\) of monic, separable, pairwise coprime quadratic polynomials in a
		chosen affine coordinate \(x\); that is, a point of \(\Adm_{\Fq}(\Fq)\). The group
		\[
		\Aff(\Fq)=\{\alpha_{a,b}\colon x\mapsto ax+b\ :\ a\in\Fq^{\times},\ b\in\Fq\},
		\qquad \#\Aff(\Fq)=q(q-1),
		\]
		acts on monic quadratics by \(s\mapsto s^{\alpha}:=a^{-2}s(\alpha(x))\) and on
		triples componentwise. By Proposition~\ref{prop:equivariance} this action
		preserves \(\Adm_{\Fq}\), \(\Dec_{\Fq}\), \(\Gfive{}_{\!,\Fq}\) and
		\(\Gsix{}_{\!,\Fq}\). We call the orbits \emph{affine-equivalence classes of
			displayed split Richelot data}.
	\end{definition}
	
	\begin{lemma}\label{lem:fix}
		Let \(\alpha=\alpha_{a,b}\) be a nontrivial element of \(\Aff(\Fq)\) and let
		\(\Fix(\alpha)\) denote the set of separable monic quadratics \(s\) with
		\(s^{\alpha}=s\).
		\begin{enumerate}[label=\textup{(\roman*)},leftmargin=*]
			\item If \(a=1\) then \(\Fix(\alpha)=\emptyset\).
			\item If \(a\neq\pm1\) then \(\#\Fix(\alpha)\le1\).
			\item If \(a=-1\), say \(\alpha=\iota_b\colon x\mapsto-x+b\), then
			\(\Fix(\iota_b)\) is the set of separable monic quadratics with middle
			coefficient \(-b\), that is the vertical line \(\{s_1=-b\}\) minus its point on
			\(\Pi\). In particular \(\#\Fix(\iota_b)=q-1\).
		\end{enumerate}
		Consequently the only nontrivial affine maps fixing three or more separable monic
		quadratics are the \(q\) involutions \(\iota_b\), \(b\in\Fq\).
	\end{lemma}
	
	\begin{proof}
		Write \(s=x^{2}+s_1x+s_0\). Expanding \(a^{-2}s(ax+b)\) gives
		\[
		s^{\alpha}=x^{2}+\frac{s_1+2b}{a}\,x+\frac{b^{2}+s_1b+s_0}{a^{2}},
		\]
		so \(s^{\alpha}=s\) if and only if
		\begin{equation}\label{eq:fixeq}
			(a-1)s_1=2b\qquad\text{and}\qquad (a^{2}-1)s_0=b^{2}+s_1b .
		\end{equation}
		If \(a=1\) the first equation forces \(b=0\), so \(\alpha\) is the identity. If
		\(a\neq1\) the first equation determines \(s_1=2b/(a-1)\) uniquely; if moreover
		\(a^{2}\neq1\), the second determines \(s_0\) uniquely. If \(a=-1\) the first
		equation reads \(s_1=-b\) and the second reads \(0=b^{2}+s_1b=0\), which holds
		identically, so \(\Fix(\iota_b)\) is cut out by \(s_1=-b\) alone. The vertical
		line \(\{s_1=-b\}\) meets \(\Pi\) in one point by Lemma~\ref{lem:census}.
	\end{proof}
	
	\begin{proposition}\label{prop:stab}
		Let \((u,v,w)\in\Adm_{\Fq}(\Fq)\) and let \(\Stab(u,v,w)\) be its stabiliser in
		\(\Aff(\Fq)\).
		\begin{enumerate}[label=\textup{(\roman*)},leftmargin=*]
			\item Either \(\Stab(u,v,w)=1\), or \(\Stab(u,v,w)=\langle\iota_b\rangle\cong C_2\)
			for a unique \(b\in\Fq\); the latter occurs exactly when \(u_1=v_1=w_1=-b\).
			\item Every triple with nontrivial stabiliser lies in \(\Dec_{\Fq}\).
			\item The converse fails. The stabilised locus is exactly the part of
			\(\Dec_{\Fq}(\Fq)\) carried by vertical lines, of cardinality
			\(q(q-1)(q-2)(q-3)\), whereas \(D_q=q(q-1)(q-2)(q^{2}-3q+3)\); the stabilised
			locus is therefore a proper sublocus of \(\Dec_{\Fq}\), of relative density
			\((q-3)/(q^{2}-3q+3)\).
		\end{enumerate}
	\end{proposition}
	
	\begin{proof}
		(i) A nontrivial \(\alpha\in\Stab(u,v,w)\) fixes the three distinct quadratics
		\(u,v,w\), so \(\alpha=\iota_b\) for some \(b\) by Lemma~\ref{lem:fix}, and then
		\(u_1=v_1=w_1=-b\); this determines \(b\). Conversely, if the three middle
		coefficients agree then \(\iota_{-u_1}\) fixes all three.
		
		(ii) Equal middle coefficients means that the three points are collinear on a
		vertical line, so \(\dcoef=0\) by Lemma~\ref{lem:dictionary}(iii).
		
		(iii) By (i) the stabilised locus consists of the ordered triples of distinct
		separable points on a vertical line, of which there are
		\(q\cdot(q-1)(q-2)(q-3)\), all admissible because any two points of a vertical
		line differ by a nonzero constant and hence are coprime. The comparison with
		\eqref{eq:Dq-split} shows that the secant and external contributions to \(D_q\)
		lie outside the stabilised locus, and both are nonzero for \(q\ge5\).
	\end{proof}
	
	\begin{remark}\label{rem:stab-normal}
		A triple with nontrivial stabiliser has \(u_1=v_1=w_1\); after the affine shift
		\(x\mapsto x-u_1/2\), which changes no predicate by
		Proposition~\ref{prop:shift}, it becomes
		\begin{equation}\label{eq:stabform}
			u=x^{2}+\beta_1,\qquad v=x^{2}+\beta_2,\qquad w=x^{2}+\beta_3
		\end{equation}
		with \(\beta_1,\beta_2,\beta_3\) distinct and nonzero, and the stabilising
		involution is \(x\mapsto-x\). Conversely every triple of the form
		\eqref{eq:stabform} with distinct nonzero \(\beta_i\) is admissible and is
		stabilised by \(x\mapsto-x\). Equation \eqref{eq:stabform} is a normal form for the stabilised locus only, not for \(\Dec_{\Fq}\); by
		Proposition~\ref{prop:stab}(iii) most decomposable triples are not of this shape.
		The curve \(y^{2}=(x^{2}+\beta_1)(x^{2}+\beta_2)(x^{2}+\beta_3)\) is the standard
		normal form for a genus-2 curve whose reduced automorphism group contains an
		involution, and \eqref{eq:stabform} is its distinguished decomposable kernel; see
		\cite[\S5, Case~(1)]{KatsuraTakashima2020}.
	\end{remark}
	
	\begin{theorem}[Affine groupoid mass, orbit count, stabilised orbit count]\label{thm:mass}
		Let \(q\ge3\) be an odd prime power. For an \(\Aff(\Fq)\)-stable subset
		\(X\subseteq\Adm_{\Fq}(\Fq)\) define the three quantities
		\begin{gather*}
			\mass_{\Aff}(X)=\sum_{\mathcal O\subseteq X/\Aff}\frac1{\#\Stab(\mathcal O)},\\
			N_{\mathrm{orb}}(X)=\#\bigl(X/\Aff\bigr),\qquad
			N_{\mathrm{stab}}(X)=\#\{\mathcal O\subseteq X/\Aff:\ \Stab(\mathcal O)\neq1\}.
		\end{gather*}
		Then \(\mass_{\Aff}(X)=\#X/\#\Aff(\Fq)\), so
		\begin{gather*}
			\mass_{\Aff}(\Adm)=(q-2)^{2}(q^{2}-q+4),\qquad
			\mass_{\Aff}(\Dec)=(q-2)(q^{2}-3q+3),\\
			\mass_{\Aff}(\Gfive)=3(q^{3}-6q^{2}+14q-13),\\
			\mass_{\Aff}(\Gsix)=q^{4}-9q^{3}+35q^{2}-71q+61,
		\end{gather*}
		and \(\mass_{\Aff}(\Adm)=\mass_{\Aff}(\Dec)+\mass_{\Aff}(\Gfive)+\mass_{\Aff}(\Gsix)\).
		Moreover
		\[
		N_{\mathrm{orb}}(\Adm)=(q-2)^{2}(q^{2}-q+4)+(q-2)(q-3),
		\qquad
		N_{\mathrm{stab}}(\Adm)=2(q-2)(q-3),
		\]
		and every stabilised orbit lies in \(\Dec_{\Fq}\).
	\end{theorem}
	
	\begin{proof}
		The identity \(\mass_{\Aff}(X)=\#X/\#\Aff(\Fq)\) is the orbit--stabiliser theorem
		summed over orbits, so the four displayed masses are the counts of
		Theorem~\ref{thm:counts} divided by \(q(q-1)\), and their additivity is that of
		Theorem~\ref{thm:strat}.
		
		For the orbit count we apply Burnside's lemma. By Lemma~\ref{lem:fix} the only
		nontrivial elements with fixed admissible triples are the \(q\) involutions
		\(\iota_b\), and the admissible triples fixed by \(\iota_b\) are the ordered
		triples of distinct separable points on the vertical line \(\{s_1=-b\}\). Any two
		distinct such points differ by a nonzero constant, so their resultant is that
		constant squared and they are automatically coprime; hence
		\begin{gather*}
			\#\bigl(\Fix(\iota_b)\cap\Adm_{\Fq}(\Fq)\bigr)=(q-1)(q-2)(q-3),\\
			\sum_{\alpha\neq1}\#\Fix(\alpha)=q(q-1)(q-2)(q-3),
		\end{gather*}
		and Burnside gives
		\[
		\begin{aligned}
			N_{\mathrm{orb}}(\Adm)&=\frac{A_q+q(q-1)(q-2)(q-3)}{q(q-1)}\\
			&=(q-2)^{2}(q^{2}-q+4)+(q-2)(q-3).
		\end{aligned}
		\]
		By Proposition~\ref{prop:stab} the triples with nontrivial stabiliser form orbits
		of length \(q(q-1)/2\), so their number of orbits is
		\(q(q-1)(q-2)(q-3)\big/\bigl(q(q-1)/2\bigr)=2(q-2)(q-3)\). The last assertion is
		Proposition~\ref{prop:stab}(ii).
	\end{proof}
	
	\begin{remark}\label{rem:not-moduli}
		The affine groupoid mass is not a count on the moduli stack of genus-2 curves. We
		quotient only by coordinate changes that preserve the chosen affine chart and the
		monic displayed-factor presentation. In particular the following are \emph{not}
		quotiented: projective transformations sending a finite branch point to infinity;
		quadratic twists; and identifications between distinct displayed splittings of
		one and the same curve, of which a genus-2 Jacobian has fifteen. The masses of
		Theorem~\ref{thm:mass} are therefore statements about affine-equivalence classes
		of displayed split Richelot data, not about isomorphism classes of curves or of
		principally polarized abelian surfaces. Two features make them appropriate invariants of the presentation. First, they are polynomials of degree four rather than six: the two factors \(q\) and \(q-1\) removed by the group are
		exactly the coordinate freedom in the displayed model. Second, the weighting by
		\(1/\#\Stab\) is the standard convention for counting objects with
		automorphisms. The ratios are unaffected:
		\(\mass_{\Aff}(\Dec)/\mass_{\Aff}(\Adm)=D_q/A_q\), and likewise for \(\Gfive\), so
		the densities of Corollary~\ref{cor:densities} hold verbatim.
	\end{remark}
	
	\section{Arithmetic of the decomposable stratum}\label{sec:split}
	
	Throughout this section \((u,v,w)\) is a square-free split input over \(\Fq\),
	\(q\) is odd, and \(\dcoef(u,v,w)=0\). Thus \(C\colon y^{2}=u(x)v(x)w(x)\) is a
	smooth genus-2 curve, and the equation \(y^{2}=UVW\) is not used as a genus-2
	model in this branch.
	
	\subsection{Pencil, involution and splitting class}
	
	\begin{proposition}\label{prop:pencil}
		Under the standing assumptions the following are equivalent:
		\begin{enumerate}[label=\textup{(\roman*)},leftmargin=*]
			\item \(\dcoef(u,v,w)=0\);
			\item the three quadratics lie in a common pencil;
			\item their degree-two pencil maps have the same deck involution in
			\(\PGL_2(\Fq)\);
			\item the three brackets \(U,V,W\) are pairwise proportional.
		\end{enumerate}
		If \(R=\Res(u,v)\), then the square class of \(R\) is independent of the pair
		chosen among \((u,v),(v,w),(w,u)\). We call this class the \emph{splitting class}
		of the configuration, and we write \(\iota\) for the common deck involution and
		\(h=M_2x^{2}+2M_1x+M_0\) for its fixed quadratic.
	\end{proposition}
	
	\begin{proof}
		The equality \(\dcoef=0\) is equivalent to linear dependence of the three
		coefficient rows; because the quadratics are monic and pairwise
		non-proportional, any two of them span the third, which gives (i)\(\iff\)(ii).
		The three pencils are therefore equal, and Lemma~\ref{lem:deck} gives the same
		deck involution for every pair and identifies each bracket with its fixed
		divisor, which gives (ii)\(\iff\)(iii)\(\iff\)(iv). Conversely, proportional
		brackets have the same fixed divisor and hence the same pencil involution, which
		forces the three quadratics into one pencil. Finally
		Lemma~\ref{lem:three-identities}(a) gives \(\Disc(U)=4\Res(v,w)\) and its cyclic
		analogues, and the brackets differ by the common proportionality factors, whose
		ratios are squares; so the resultant square class is independent of the pair.
	\end{proof}
	
	The splitting class has a purely incidence-geometric description. It is the entry in the last column of Table~\ref{tab:ledger}, and it
	refines the count \eqref{eq:Dq-split} of \(\Dec\) into two explicit halves.
	
	\begin{proposition}\label{prop:class-line}
		Let \(P,Q,P'\in S\) be three distinct collinear points on a non-tangent line
		\(L\), so that the corresponding triple lies in \(\Dec\).
		\begin{enumerate}[label=\textup{(\roman*)},leftmargin=*]
			\item If \(L\) is the vertical line \(a=c\), then \(\Res(P,Q)\) is a nonzero
			square, so the splitting class is a square.
			\item If \(L\) is the non-vertical line \(b=ma+c\), and \(P=(a,b)\),
			\(Q=(a',b')\), then
			\begin{equation}\label{eq:res-line}
				\Res(P,Q)=(a'-a)^{2}\,(m^{2}+c).
			\end{equation}
			Hence the splitting class is the square class of \(m^{2}+c\): it is a square when
			\(L\) is secant to \(\Pi\) and a nonsquare when \(L\) is external.
		\end{enumerate}
	\end{proposition}
	
	\begin{proof}
		For \(P=(a,b)\) and \(Q=(a',b')\) the minors of the pair are \(M_2=a'-a\),
		\(M_1=b'-b\), \(M_0=ab'-a'b\), and \(\Res(P,Q)=M_1^{2}-M_2M_0\) by
		Lemma~\ref{lem:bezoutian}. If \(L\) is vertical then \(a=a'\), so \(M_2=0\) and
		\(\Res(P,Q)=M_1^{2}=(b'-b)^{2}\), a nonzero square since \(P\neq Q\). If
		\(L\colon b=ma+c\) then \(b'-b=m(a'-a)\) and
		\[
		ab'-a'b=a(ma'+c)-a'(ma+c)=c(a-a'),
		\]
		so \(\Res(P,Q)=m^{2}(a'-a)^{2}-(a'-a)\cdot c(a-a')=(a'-a)^{2}(m^{2}+c)\), which is
		\eqref{eq:res-line}. The classification of \(L\) by the square class of
		\(m^{2}+c\) is the proof of Lemma~\ref{lem:census}.
	\end{proof}
	
	\begin{corollary}\label{cor:count-class}
		Write \(D_q^{\square}\) and \(D_q^{\not\square}\) for the numbers of ordered
		admissible triples with \(\dcoef=0\) and splitting class a square, respectively a
		nonsquare. Then
		\[
		D_q^{\square}=\tfrac12\,q(q-1)(q-2)^{2}(q-3),
		\qquad
		D_q^{\not\square}=\tfrac12\,q^{2}(q-1)^{2}(q-2),
		\]
		and \(D_q^{\square}+D_q^{\not\square}=D_q\). In particular
		\[
		\frac{D_q^{\not\square}}{D_q}=\frac{q(q-1)}{2(q^{2}-3q+3)}=\frac12+\frac1q+O(q^{-2}),
		\]
		so the two cases occur in asymptotically equal proportions, and the stabilised
		locus of Proposition~\ref{prop:stab} is contained in the square case.
	\end{corollary}
	
	\begin{proof}
		By Proposition~\ref{prop:class-line} the square case is exactly the union of the
		vertical and secant contributions to \eqref{eq:Dq-split}, and the nonsquare case
		is the external contribution. Adding the first two terms of
		\eqref{eq:Dq-split} gives
		\[
		\begin{aligned}
			q(q-1)(q-2)(q-3)&+\tfrac12 q(q-1)(q-2)(q-3)(q-4)\\
			&=\tfrac12 q(q-1)(q-2)(q-3)\bigl(2+(q-4)\bigr)\\
			&=\tfrac12 q(q-1)(q-2)^{2}(q-3),
		\end{aligned}
		\]
		and the third term is \(\tfrac12 q^{2}(q-1)^{2}(q-2)\). The containment of the
		stabilised locus is Proposition~\ref{prop:stab}(iii) together with
		Proposition~\ref{prop:class-line}(i).
	\end{proof}
	
	\subsection{The extra involution}
	
	The deck involution of the pencil is not merely an auxiliary object: it lifts to
	an automorphism of \(C\). This is what makes the branch \(\dcoef=0\) a statement
	about the curve rather than about the presentation.
	
	\begin{proposition}\label{prop:extra-involution}
		Let \((u,v,w)\) be a square-free split input with \(\dcoef=0\), let
		\(R\) be the splitting class and \((M_2,M_1,M_0)\) the minors of \((u,v)\), and
		let \(c\) satisfy \(c^{2}=R^{3}\). Then
		\begin{equation}\label{eq:tau}
			\tau\colon (x,y)\longmapsto\Bigl(\iota(x),\ \frac{c\,y}{(M_2x+M_1)^{3}}\Bigr),
			\qquad \iota(x)=\frac{-M_1x-M_0}{M_2x+M_1},
		\end{equation}
		is an automorphism of \(C\) with \(\tau^{2}=\mathrm{id}\), and \(\tau\) is not the
		hyperelliptic involution. Consequently the reduced automorphism group of \(C\)
		contains \(\langle\iota\rangle\cong C_2\). Moreover:
		\begin{enumerate}[label=\textup{(\roman*)},leftmargin=*]
			\item \(\tau\) is defined over \(\Fq\) if and only if the splitting class is a
			square; otherwise it is defined over \(\F_{q^{2}}\);
			\item \(\iota\) is an affine map of the \(x\)-line, equivalently fixes the point at
			infinity, if and only if \(M_2=0\) for one and hence for all three pairs,
			equivalently \(u_1=v_1=w_1\), that is, exactly on the stabilised locus of
			Proposition~\ref{prop:stab}.
		\end{enumerate}
	\end{proposition}
	
	\begin{proof}
		By Proposition~\ref{prop:pencil} the three factors lie in the pencil spanned by
		\(u\) and \(v\), so \eqref{eq:scaling} applies to each of them and multiplying the
		three instances gives
		\begin{equation}\label{eq:f-scaling}
			f\bigl(\iota(x)\bigr)\,(M_2x+M_1)^{6}=R^{3}f(x).
		\end{equation}
		If \(c^{2}=R^{3}\), then for \((x,y)\in C\),
		\[
		\Bigl(\frac{cy}{(M_2x+M_1)^{3}}\Bigr)^{2}
		=\frac{R^{3}f(x)}{(M_2x+M_1)^{6}}=f\bigl(\iota(x)\bigr),
		\]
		so \(\tau\) maps \(C\) to \(C\). A direct computation gives
		\[
		M_2\iota(x)+M_1=\frac{M_1^{2}-M_2M_0}{M_2x+M_1}=\frac{R}{M_2x+M_1},
		\]
		whence \((M_2x+M_1)^{3}\bigl(M_2\iota(x)+M_1\bigr)^{3}=R^{3}=c^{2}\) and therefore
		\(\tau^{2}=\mathrm{id}\). Since \(\iota\neq\mathrm{id}\), \(\tau\) acts
		nontrivially on \(x\) and is not the hyperelliptic involution.
		
		(i) The equation \(c^{2}=R^{3}=R\cdot R^{2}\) has a solution in \(\Fq\) exactly
		when \(R\) is a square, and always has one in \(\F_{q^{2}}\).
		
		(ii) The map \eqref{eq:iota} is affine exactly when \(M_2=0\). By
		Proposition~\ref{prop:pencil} the three brackets are proportional, so their
		leading coefficients vanish simultaneously; and for monic quadratics the leading
		minor of a pair is the difference of the middle coefficients
		(Lemma~\ref{lem:aligned}), so \(M_2=0\) for all three pairs is exactly
		\(u_1=v_1=w_1\). This is the condition of Proposition~\ref{prop:stab}(i).
	\end{proof}
	
	\begin{remark}
		\label{rem:bolza}
		Bolza \cite[\S11]{Bolza1887} attached to a binary sextic an invariant, expressible
		as a \(3\times3\) determinant in the Clebsch invariants, whose vanishing
		characterises the genus-2 curves carrying an extra involution; see
		\cite[Prop.~4.3]{KatsuraTakashima2020} for the equivalence, at the level of the
		reduced automorphism group, between the existence of an involution and the
		existence of a decomposed Richelot isogeny. Proposition~\ref{prop:extra-involution}
		proves one direction of the corresponding coefficient-level statement: a
		displayed splitting with \(\dcoef=0\) produces an explicit extra involution. We do
		not prove the converse here, and in particular we do not claim that
		\(\dcoef=0\) for one of the fifteen quadratic splittings of a given sextic is
		equivalent to the vanishing of Bolza's invariant; that equivalence is a
		moduli-level statement whose \(\Fq\)-rational form would require a separate
		treatment of the Galois action on the fifteen splittings (\S\ref{sec:limits}).
		What the present section does assert is that \(\dcoef\) is a local,
		splitting-dependent counterpart of that global invariant, computable in
		\(2\mathrm{M}\).
	\end{remark}
	
	\subsection{Explicit elliptic factors}
	
	\begin{lemma}\label{lem:fixed-avoid}
		Under the standing assumptions, neither fixed point of \(\iota\) is a root of
		\(f=uvw\).
	\end{lemma}
	
	\begin{proof}
		Let \(\xi\) be a root of the common fixed quadratic \(h\); by
		Lemma~\ref{lem:deck} and Proposition~\ref{prop:pencil}(iv), \(h\) is proportional
		to each of the three brackets. If \(u(\xi)=0\), use \(h\uptounit W=u'v-uv'\):
		then \(W(\xi)=u'(\xi)v(\xi)\), and \(v(\xi)\neq0\) by coprimality, so
		\(W(\xi)=0\) forces \(u'(\xi)=0\), a double root of \(u\), contradicting
		\(\Disc(u)\neq0\). The case \(v(\xi)=0\) is symmetric. If \(w(\xi)=0\), use
		\(h\uptounit U=v'w-vw'\): then \(U(\xi)=-v(\xi)w'(\xi)=0\) forces \(v(\xi)=0\),
		excluded, or \(w'(\xi)=0\), contradicting \(\Disc(w)\neq0\).
	\end{proof}
	
	\begin{theorem}\label{thm:split}
		Let \(R\) be the splitting class and let \(K/\Fq\) be a field over which the two
		fixed points \(\alpha,\beta\in\PP^{1}(K)\) of \(\iota\) are rational. Assume,
		after a projective change of coordinate over \(K\), that \(\alpha\) and \(\beta\)
		are finite; by Lemma~\ref{lem:fixed-avoid}, \(f(\alpha)f(\beta)\neq0\)
		automatically. Put
		\[
		T=\frac{x-\alpha}{x-\beta},\qquad\text{so}\qquad x(T)=\frac{\alpha-\beta T}{1-T}.
		\]
		Then, after multiplying by a nonzero scalar in \(K\), the transformed sextic has
		the form
		\begin{gather*}
			\tilde f(T):=(1-T)^{6}f\bigl(x(T)\bigr)=F(T^{2}),\\
			F(X)=f_3X^{3}+f_2X^{2}+f_1X+f_0,\qquad f_0f_3\neq0 .
		\end{gather*}
		The curves
		\[
		E_1\colon Y^{2}=F(X),\qquad E_2\colon Z^{2}=f_0X^{3}+f_1X^{2}+f_2X+f_3
		\]
		are elliptic curves over \(K\), the maps
		\[
		\pi_1\colon(T,\tilde y)\mapsto(T^{2},\tilde y),\qquad
		\pi_2\colon(T,\tilde y)\mapsto(T^{-2},\tilde yT^{-3})
		\]
		are degree-two morphisms from \(C_K\) to \(E_1\) and \(E_2\), and the induced
		homomorphism
		\[
		(\pi_{1,*},\pi_{2,*})\colon\Jac(C_K)\longrightarrow E_1\times E_2
		\]
		is the Richelot quotient associated with the displayed splitting.
	\end{theorem}
	
	\begin{proof}
		In the coordinate \(T\) the involution \(\iota\) is conjugate to
		\(\sigma\colon T\mapsto-T\). Each member of the common pencil is invariant under
		\(\sigma\) up to a scalar by \eqref{eq:scaling}, and since none of \(u,v,w\)
		vanishes at a fixed point (Lemma~\ref{lem:fixed-avoid}) that scalar is \(+1\);
		hence every factor, and therefore the transformed sextic, is an even polynomial
		in \(T\) with nonzero constant and leading coefficients. Square-freeness of \(f\)
		implies that both cubics are separable. The displayed formulas for \(\pi_1\) and
		\(\pi_2\) follow from the equations, and their degrees are two: they are the
		quotients by \(\sigma\) and by \(\iota_C\sigma\), with \(\iota_C\) the
		hyperelliptic involution.
		
		It remains to identify \(\ker(\pi_{1,*},\pi_{2,*})\) with the Richelot kernel
		\(G=\{0,[u],[v],[w]\}\). In the coordinate \(T\) each of \(\tilde u,\tilde v,
		\tilde w\) is even, so its two roots form a pair \(\{\rho,-\rho\}\). Under
		\(\pi_1\) both map to the single point \((\rho^{2},0)\in E_1[2]\), while the two
		points of \(C_K\) above \(T=\infty\) map to \(O_{E_1}\); hence
		\(\pi_{1,*}[\tilde u]=2\bigl[(\rho^{2},0)-O_{E_1}\bigr]=0\). Under \(\pi_2\) the
		same pair maps to \((\rho^{-2},0)\in E_2[2]\) and the two points above
		\(T=\infty\) map to a pair of mutually inverse points, so
		\(\pi_{2,*}[\tilde u]=0\) as well. The same computation applies to \(\tilde v\)
		and \(\tilde w\), whence \(G\subseteq\ker(\pi_{1,*},\pi_{2,*})\). By Kuhn
		\cite[\S1]{Kuhn1988}, for a genus-2 curve carrying two complementary optimal
		degree-2 elliptic subcovers the map \((\pi_{1,*},\pi_{2,*})\) is an isogeny of
		degree 4; since \(\#G=4\), the inclusion is an equality.
	\end{proof}
	
	\begin{corollary}[Field of definition]\label{cor:field}
		Let \(R\) be the splitting class.
		\begin{enumerate}[label=\textup{(\roman*)},leftmargin=*]
			\item If \(R\) is a square in \(\Fq\), the ordered pair \((E_1,E_2)\) of
			Theorem~\ref{thm:split} is defined over \(\Fq\).
			\item If \(R\) is a nonsquare, the construction is defined over \(\F_{q^{2}}\)
			and the nontrivial element of \(\mathrm{Gal}(\F_{q^{2}}/\Fq)\) exchanges the two
			labelled factors together with their covers.
		\end{enumerate}
		In either case the Richelot quotient itself, that is the abelian surface
		\(\Jac(C)/G\) with its induced principal polarization, is defined over \(\Fq\),
		because the kernel \(G=\{0,[u],[v],[w]\}\) is \(\Fq\)-rational. In case (ii) no
		single labelled elliptic factor is an \(\Fq\)-rational output; the Galois-stable
		unordered factor datum is.
	\end{corollary}
	
	\begin{proof}
		The fixed points of \(\iota\) are rational exactly when \(h\) splits, that is
		when \(\Disc(h)=4R\) has square class that of \(R\)
		(Lemma~\ref{lem:three-identities}(a)). In the nonsplit case they become rational
		over \(\F_{q^{2}}\) and Galois interchanges them; this replaces \(T\) by
		\(T^{-1}\), which reverses the cubic \(F\) and exchanges \(E_1\) and \(E_2\).
		Rationality of the quotient is standard: \(G\) is stable under the \(q\)-power
		Frobenius because the three quadratics are, so the isogeny \(\Jac(C)\to\Jac(C)/G\)
		and the polarization induced by \(2\Theta\) are defined over \(\Fq\).
	\end{proof}
	
	\begin{remark}\label{rem:no-descent}
		We do not claim an effective descent datum exhibiting the quotient as a product
		over \(\Fq\) in case (ii). The statements used below concern the Weil polynomial
		of \(\Jac(C)\) and the Galois-stable unordered factor datum, both of which are
		\(\Fq\)-rational, and both of which are established directly in
		Corollary~\ref{cor:weil} without any descent theorem.
	\end{remark}
	
	\subsection{Frobenius data}
	
	We use the reciprocal normalization of the Weil polynomial: for an abelian
	variety \(A\) over \(\Fq\) with \(q\)-power Frobenius \(\pi\),
	\[
	P_A(T)=\det\bigl(1-T\pi\mid V_\ell A\bigr)=\prod_i(1-\alpha_iT),
	\]
	so that \(\#A(\Fq)=P_A(1)\), and for a curve \(C\) of genus 2,
	\(\#C(\Fq)=q+1-\sum_i\alpha_i\), that is, the coefficient of \(T\) in
	\(P_{\Jac(C)}\) equals \(-(q+1-\#C(\Fq))\).
	
	\begin{lemma}\label{lem:frob-exchange}
		Let \(A\) be an abelian surface over \(\Fq\) such that over \(K=\F_{q^{2}}\) one
		has \(A_K\sim E_1\times E_2\) with the \(q\)-power Frobenius \(\pi\) interchanging
		the two factors. Then in a basis of \(V_\ell A\) adapted to the decomposition,
		\(\pi\) is represented by a block matrix \(\left(\begin{smallmatrix}0&B\\A&0\end{smallmatrix}\right)\)
		with \(2\times2\) blocks, and
		\[
		P_A(T)=\det\bigl(1-T^{2}BA\bigr)=P_{E_1/\F_{q^{2}}}(T^{2}).
		\]
		In particular \(P_A\) is even, and \(\#A(\Fq)=\#E_1(\F_{q^{2}})\).
	\end{lemma}
	
	\begin{proof}
		The block shape is the statement that \(\pi\) maps the \(E_1\)-part to the
		\(E_2\)-part and back. For such a matrix,
		\[
		\det\begin{pmatrix}1&-TB\\-TA&1\end{pmatrix}=\det\bigl(1-T^{2}BA\bigr),
		\]
		and \(BA\) represents \(\pi^{2}\), the \(q^{2}\)-power Frobenius, on the
		\(E_1\)-part. Evaluating at \(T=1\) gives \(\#A(\Fq)=P_{E_1/\F_{q^{2}}}(1)=\#E_1(\F_{q^{2}})\).
	\end{proof}
	
	\begin{corollary}\label{cor:weil}
		Let \((u,v,w)\) be a square-free split input with \(\dcoef=0\) and let
		\(P_{\Jac(C)}\) be the Weil polynomial of \(\Jac(C)\) over \(\Fq\).
		\begin{enumerate}[label=\textup{(\alph*)},leftmargin=*]
			\item If the splitting class is a square, then \(E_1,E_2\) are defined over
			\(\Fq\), \(\Jac(C)\) is \(\Fq\)-isogenous to \(E_1\times E_2\) through the
			quotient by \(G\), and with \(P_{E_i}(T)=1-a_iT+qT^{2}\),
			\[
			P_{\Jac(C)}(T)=1-(a_1+a_2)T+(a_1a_2+2q)T^{2}-q(a_1+a_2)T^{3}+q^{2}T^{4}.
			\]
			\item If the splitting class is a nonsquare, then with \(a\) the trace of the
			\(q^{2}\)-power Frobenius of \(E_1\), so that
			\(P_{E_1/\F_{q^{2}}}(T)=1-aT+q^{2}T^{2}\),
			\[
			P_{\Jac(C)}(T)=1-aT^{2}+q^{2}T^{4},
			\qquad
			\#\Jac(C)(\Fq)=\#E_1(\F_{q^{2}}).
			\]
		\end{enumerate}
	\end{corollary}
	
	\begin{proof}
		(a) The kernel \(G\) is \(\Fq\)-rational, so the quotient map is defined over
		\(\Fq\) and \(\Jac(C)\) is \(\Fq\)-isogenous to the product by
		Theorem~\ref{thm:split} and Corollary~\ref{cor:field}(i). Weil polynomials
		multiply on products and are invariant under \(\Fq\)-isogeny.
		(b) By Corollary~\ref{cor:field}(ii) the \(q\)-power Frobenius exchanges the two
		factors, so Lemma~\ref{lem:frob-exchange} applies with \(A=\Jac(C)\).
	\end{proof}
	
	\begin{corollary}\label{cor:pointcount}
		If \((u,v,w)\) is a square-free split input with \(\dcoef=0\) whose splitting
		class is a nonsquare, then \(\#C(\Fq)=q+1\).
	\end{corollary}
	
	\begin{proof}
		By Corollary~\ref{cor:weil}(b) the coefficient of \(T\) in \(P_{\Jac(C)}\)
		vanishes, and that coefficient is \(-(q+1-\#C(\Fq))\).
	\end{proof}
	
	\begin{remark}\label{rem:predicate}
		Corollary~\ref{cor:pointcount} gives a necessary condition for the nonsplit
		decomposable branch that is computed from \(C\) alone, without reference to any
		quadratic splitting. It is therefore an independent consistency check on an
		implementation of Algorithm~\ref{alg:richelot}: whenever the evaluator reports
		\textsc{Decomposable} with a nonsquare splitting class, a point count on \(C\)
		must return exactly \(q+1\). We use it in that role in \S\ref{sec:verify}. By
		Corollary~\ref{cor:count-class} the branch it certifies carries
		\(\tfrac12+O(q^{-1})\) of the decomposable stratum, so the check is not vacuous.
	\end{remark}
	
	\begin{remark}\label{rem:weilres}
		Corollary~\ref{cor:weil}(b) says that \(\Jac(C)\) is \(\Fq\)-isogenous to the Weil
		restriction of \(E_1\) along \(\F_{q^{2}}/\Fq\), and
		\(\#\Jac(C)(\Fq)=\#E_1(\F_{q^{2}})\) is the identity that makes point counting on
		such a Jacobian as cheap as point counting on an elliptic curve over
		\(\F_{q^{2}}\). Genus-2
		curves with this property go back to a 2003 preprint of Scholten; Bernstein and
		Lange describe them under the name Scholten curves and analyse their rarity
		\cite[\S\S2.1 and~6]{BernsteinLange2014}. We use
		the name only in this historical remark; the statements of
		Corollaries~\ref{cor:weil} and~\ref{cor:pointcount} are proved above and do not
		depend on it. What the present section adds to that literature is a criterion on
		the displayed splitting, decided in \(2\mathrm{M}\), together with the exact
		frequency of the configuration in the parameter space
		(Corollary~\ref{cor:count-class}).
	\end{remark}
	
	\begin{remark}\label{rem:split-output}
		In the branch \(\dcoef=0\) the evaluator returns the common pencil, equivalently
		its deck involution and fixed quadratic, together with the splitting class and the
		unordered elliptic-factor datum of Theorem~\ref{thm:split}. When the splitting
		class is a nonsquare, that datum is represented over \(\F_{q^{2}}\) with
		Galois-conjugate factors, and an implementation must not present one factor as an
		\(\Fq\)-rational output. What it may present over \(\Fq\) is the quotient surface
		with its polarization and quotient map (Corollary~\ref{cor:field}), the Weil
		polynomial (Corollary~\ref{cor:weil}), and every symmetric function of the pair. Equivalently, the branch output is
		organised over the quadratic \'etale algebra \(\Fq[t]/(h(t))\): the algebra
		is split exactly when the splitting class is a square, and its two geometric
		points are the two fixed points of \(\iota\).
	\end{remark}
	
	\begin{remark}\label{rem:split-credit}
		Complementary degree-two elliptic covers and the polarized \((2,2)\)-split
		Jacobian are classical \cite[\S1]{Kuhn1988}\cite[\S8]{Smith2005}. This section adds arithmetic over the base field is arithmetic over the
		base field: the identification of the splitting class with a line invariant and
		the resulting refined count (Proposition~\ref{prop:class-line},
		Corollary~\ref{cor:count-class}), the explicit extra involution and its two
		rationality criteria (Proposition~\ref{prop:extra-involution}), the field of
		definition of the factors (Corollary~\ref{cor:field}), the two Frobenius shapes
		(Corollary~\ref{cor:weil}), the point-count identity
		(Corollary~\ref{cor:pointcount}), and the certification of the branch by
		\(\dcoef=0\) at a cost of \(2\mathrm{M}\).
	\end{remark}
	
	\section{Cost model, operation counts and verification}\label{sec:verify}
	
	\subsection{Cost model}\label{subsec:cost}
	
	Costs are measured in a monic quadratic coefficient model. A field multiplication
	costs \(\mathrm{M}\) and a field squaring costs \(\mathrm{S}\); additions,
	subtractions, negations, comparisons with zero, and multiplications by fixed
	constants, in particular by 2 and by 4, are not charged. Inversions, quadratic
	residuosity tests, square roots, and coordinate changes are charged separately
	and are never folded into a count of \(\mathrm{M}\) and \(\mathrm{S}\). No
	asymptotic claim is made or intended: every degree-two Richelot realization has
	constant cost per step, so the only meaningful comparison is a count of named
	operations in a declared model.
	
	\subsection{Operation counts}\label{subsec:opcounts}
	
	For monic input, each bracket is \(M_2x^{2}+2M_1x+M_0\) with
	\(M_2=b_1-a_1\), \(M_1=b_0-a_0\) free and \(M_0=a_1b_0-a_0b_1\) costing
	\(2\mathrm{M}\). Table~\ref{tab:cost} separates detection of the branch from
	construction of its output; conflating the two is the reason a figure such as
	``the decomposable branch costs \(2\mathrm{M}\)'' can be misread.
	
	\begin{table}[h!]
		\centering
		\caption{Operation counts in the cost model of \S\ref{subsec:cost}.}
		\label{tab:cost}
		\small
		\setlength{\tabcolsep}{4pt}
		\begin{tabular}{@{}>{\raggedright\arraybackslash}p{0.42\linewidth}
				>{\raggedright\arraybackslash}p{0.25\linewidth}
				>{\raggedright\arraybackslash}p{0.26\linewidth}@{}}
			\toprule
			Task & Cost & Output\\
			\midrule
			Three brackets from the minors & \(6\mathrm{M}\) & \((U,V,W)\)\\
			Three input discriminants \(s_1^{2}-4s_0\) & \(3\mathrm{S}\) & predicate \(C_1\)\\
			Three pairwise resultants \(M_1^{2}-M_2M_0\) & \(3\mathrm{M}+3\mathrm{S}\) & predicate \(C_2\)\\
			Branch determinant, Proposition~\ref{prop:d-syzygy} & \(2\mathrm{M}\) & predicate \(C_3\), the tag\\
			\midrule
			Full certificate beyond the brackets & \(5\mathrm{M}+6\mathrm{S}\) & tag \(+\) validity\\
			Branch determinant alone, on a certified chain & \(2\mathrm{M}\) & tag\\
			\midrule
			Normalize the genus-2 codomain \(y^{2}=\dcoef^{-1}UVW\) & 1 inversion, batched & monic model\\
			Materialize the decomposable datum (Theorem~\ref{thm:split}) & 1 residuosity test, 1 square root in \(\Fq\) or \(\F_{q^{2}}\), \(O(1)\) coordinate change & involution, splitting class, elliptic factors with covers\\
			Continue a chain through a quintic output & 1 projective substitution & new displayed splitting\\
			\bottomrule
		\end{tabular}
	\end{table}
	
	No square root and no inversion is intrinsic to the certification layer. Along a
	chain in which each previous output has been certified, the square-freeness data
	of the next displayed splitting are inherited and the only new predicate is
	\(\dcoef\), at \(2\mathrm{M}\) on monic input. What distinguishes the certified
	evaluator is not a lower operation count but that it classifies the quotient from
	data it computes anyway, accepts the valid degree-five locus without a change of
	model, and returns structured arithmetic when \(\dcoef=0\); its extra work
	replaces downstream failure handling rather than adding a validation pass.
	
	\subsection{Symbolic verification}\label{subsec:symbolic}
	
	The identities behind the certificate are polynomial identities in the six input
	coefficients and are checked once over the integral polynomial ring
	\(\Z[u_0,u_1,v_0,v_1,w_0,w_1]\), hence simultaneously for all odd finite fields.
	The following equalities are the core symbolic checks:
	\begin{itemize}[leftmargin=*,itemsep=1pt]
		\item \(a'b-ab'=M_2x^{2}+2M_1x+M_0\) (Lemma~\ref{lem:bracket-minors}), which
		validates the derivative-free bracket assembly;
		\item \(\Res(a,b)=M_1^{2}-M_2M_0\) and \(\Disc(a'b-ab')=4\Res(a,b)\)
		(Lemmas~\ref{lem:bezoutian} and~\ref{lem:three-identities}(a)), which validate the
		coprimality predicate;
		\item the minor syzygy \(a_2M_0-a_1M_1+a_0M_2=0\) (Proposition~\ref{prop:syzygy})
		and \(\dcoef=-(u_0M_2-u_1M_1+u_2M_0)\) (Proposition~\ref{prop:d-syzygy}), which
		validate the branch classifier;
		\item \(U'V-UV'=-2\dcoef\,w\) and its two cyclic analogues, and
		\(\dcoef(U,V,W)=-2\dcoef(u,v,w)^{2}\) (Lemma~\ref{lem:three-identities}(b),(c));
		\item invariance of \(\Disc\), \(\Res\) and \(\dcoef\) under \(x\mapsto x+\delta\)
		(Proposition~\ref{prop:shift}), and the translation rule of
		Lemma~\ref{lem:translate};
		\item \(\dcoef=(a-e)(b-c)\) on the alignment normal form
		(Lemma~\ref{lem:align-normal});
		\item the pencil scaling identity \eqref{eq:scaling} for the members of the
		pencil, its sextic consequence \eqref{eq:f-scaling} on the parametrization
		\(w=u+\lambda(v-u)\), and the line form \eqref{eq:res-line} of the resultant
		(Lemma~\ref{lem:deck}, Propositions~\ref{prop:extra-involution}
		and~\ref{prop:class-line}).
	\end{itemize}
	Establishing these over \(\Z\) rather than modulo a prime covers all odd finite fields at once and prevents a finite-sample test from being mistaken for a proof. Proposition~\ref{prop:postcheck} then shows that no further algebraic predicate on the output is needed once \(C_1\)--\(C_3\) hold.
	
	\subsection{Exhaustive layers}\label{subsec:exhaustive}
	
	The computational supplement performs the following complete enumerations. No
	random sampling is involved, and none of them is used in a proof; they are
	independent regression checks of the closed formulas, of the case split and of
	the implementation.
	
	\paragraph{Counts and masses}
	The ordered stratum counts \eqref{eq:Aq}--\eqref{eq:G6q}, the refined counts
	\(D_q^{\square}\) and \(D_q^{\not\square}\) of Corollary~\ref{cor:count-class},
	the number of stabilised triples, the orbit count
	\(N_{\mathrm{orb}}(\Adm)\) and the stabilised orbit count
	\(N_{\mathrm{stab}}(\Adm)\) of Theorem~\ref{thm:mass} were verified by exhaustive
	enumeration of \(\Adm_{\Fq}(\Fq)\) for every odd prime power
	\[
	q\in\{3,5,7,9,11,13,17,19,23,25,27,29,31,37,41,43,47,49\}.
	\]
	The non-prime values matter: the theorems are stated for all odd prime powers,
	while the implementation discussion of \S\ref{subsec:opcounts} is restricted to
	prime fields. The orbit data were obtained in two independent ways: by direct
	enumeration of the \(\Aff(\Fq)\)-orbits for \(q\le11\), and, for the whole range,
	by computing \(\#\bigl(\Fix(\alpha)\cap\Adm_{\Fq}(\Fq)\bigr)\) for every
	\(\alpha\in\Aff(\Fq)\) and feeding the result to Burnside's lemma. The second
	computation also confirms Lemma~\ref{lem:fix} directly: for every \(q\) in the
	range and every nontrivial \(\alpha\) other than the \(q\) involutions
	\(\iota_b\), the set of admissible triples fixed by \(\alpha\) is empty.
	Table~\ref{tab:verify} reproduces three representative fields. At
	\(q=25\), for instance, the decomposable and quintic strata carry
	\(3.98\%\) and \(11.47\%\) of the admissible triples, in line with the
	leading terms \(1/q\) and \(3/q\) of Corollary~\ref{cor:densities}.
	
	\begin{table}[H]
		\centering
		\caption{Exhaustive verification of Theorems~\ref{thm:counts}
			and~\ref{thm:mass} and of Corollary~\ref{cor:count-class}, reported as ordered
			counts. The enumeration covers every odd prime power \(q\le49\); three fields are
			displayed.}
		\label{tab:verify}
		\small
		\begin{tabular}{@{}rrrrrrr@{}}
			\toprule
			\(q\) & \(A_q\) & \(D_q\) & \(G_{5,q}\) & \(D_q^{\square}\) & \(N_{\mathrm{orb}}\) & \(N_{\mathrm{stab}}\)\\
			\midrule
			11 & 1\,015\,740 & 90\,090 & 246\,180 & 35\,640 & 9\,306 & 144\\
			13 & 3\,020\,160 & 228\,228 & 632\,736 & 94\,380 & 19\,470 & 220\\
			25 & 191\,709\,600 & 7\,631\,400 & 21\,981\,600 & 3\,491\,400 & 320\,022 & 1\,012\\
			\bottomrule
		\end{tabular}
	\end{table}
	
	\paragraph{The certificate}
	Proposition~\ref{prop:postcheck} was verified exhaustively for every odd prime
	\(p\le17\): among all ordered square-free split inputs with \(\dcoef\neq0\), of
	which there are \(15\,907\,920\) at \(p=17\), the output triple in every case has
	the asserted degree pattern, pairwise coprimality and nonvanishing discriminants,
	with no exception. Proposition~\ref{prop:shift} needs no enumeration; it is one
	of the symbolic identities of \S\ref{subsec:symbolic}.
	
	\paragraph{The decomposable branch}
	For every odd prime power \(q\le13\) the supplement enumerates all decomposable
	admissible triples and checks, in each case: that the splitting class is
	independent of the pair (Proposition~\ref{prop:pencil}); that it agrees with the
	square class of the line invariant \(m^{2}+c\), respectively is a square on a
	vertical line (Proposition~\ref{prop:class-line}); that the transformed sextic of
	Theorem~\ref{thm:split} is even with \(f_0f_3\neq0\); that
	\(\#\Jac(C)(\Fq)=\#E_1(\Fq)\cdot\#E_2(\Fq)\) when the splitting class is a
	square; and that \(\#C(\Fq)=q+1\) with an even Weil polynomial when it is a
	nonsquare (Corollaries~\ref{cor:weil} and~\ref{cor:pointcount}). The degenerate
	configurations with \(M_2=0\), where one fixed point of the pencil involution lies
	at infinity and the coordinate change of Theorem~\ref{thm:split} is a
	translation, are checked separately for every odd prime \(p\le17\), the case \(p=17\) by a
	dedicated enumeration of the vertical stratum. The identity \eqref{eq:f-scaling}
	underlying Proposition~\ref{prop:extra-involution} is one of the symbolic
	identities of \S\ref{subsec:symbolic} and therefore holds over every odd \(\Fq\);
	in addition, its pointwise evaluation at every \(x\in\F_9\), together with
	\(\iota\circ\iota=\mathrm{id}\) away from the pole and the rationality criterion
	of Proposition~\ref{prop:extra-involution}(i), is re-verified for every
	decomposable triple over \(\F_9\), and in the specialized form \(f(-x-a)=f(x)\)
	for every vertical triple at \(p=17\). The worked examples of
	\ref{app:examples} over \(\F_{101}\) are re-verified by the same code
	path.
	
	\subsection{Reproducibility}\label{subsec:repro}
	The computational supplement contains: a symbolic script establishing the
	identities of \S\ref{subsec:symbolic} over
	\(\Z[u_0,u_1,v_0,v_1,w_0,w_1]\); an implementation of
	Algorithm~\ref{alg:richelot} returning the three tags of
	Theorem~\ref{thm:certificate}; the exhaustive enumerations of
	\S\ref{subsec:exhaustive}, with the field ranges stated there and with the
	explicit construction of \(\F_9\), \(\F_{25}\), \(\F_{27}\) and \(\F_{49}\); the orbit and
	stabiliser enumerations; the raw tables from which
	Table~\ref{tab:verify} is drawn; a test file reproducing the three worked
	examples of \ref{app:examples}; and a machine-readable record of the
	software versions and command lines that regenerate every table. No wall-clock
	timing claim is made: a timing comparison would require an archived revision of
	each implementation, an explicit baseline and a stated statistical protocol, and
	until such a package exists the operation counts of \S\ref{subsec:opcounts} are
	the reproducible performance statement of this work.
	
	\section{Scope and limitations}\label{sec:limits}
	
	Five boundaries of this treatment deserve explicit statement.
	
	\paragraph{Galois-permuted kernels}
	An \(\Fq\)-rational maximal 2-Weil-isotropic kernel corresponds to a partition of
	the six Weierstrass points into three Galois-stable sets of pairs; Galois may
	permute the pairs without fixing them, in which case the quadratic factors are
	defined only over \(\F_{q^{2}}\) or \(\F_{q^{3}}\) although the kernel and the
	isogeny are defined over \(\Fq\). The classification of this paper covers the
	factor-wise rational case. The branch predicate itself survives descent: by
	Proposition~\ref{prop:equivariance} a permutation of \((u,v,w)\) changes
	\(\dcoef\) by its sign, so for a Galois-stable splitting \(\sigma(\dcoef)=\pm\dcoef\)
	for every \(\sigma\), and \(\dcoef^{2}\in\Fq\) even when the factors are not
	individually rational. The full stratification by Galois type is left open.
	
	\paragraph{Moduli-level statements}
	The counts and masses of \S\ref{sec:counts} and \S\ref{sec:mass} are counts of
	displayed data, quotiented at most by affine coordinate changes
	(Remark~\ref{rem:not-moduli}). We prove that \(\dcoef=0\) produces an extra
	involution (Proposition~\ref{prop:extra-involution}), but we do not prove the
	converse at the level of moduli, and we make no claim about Bolza's invariant of
	the sextic beyond the pointer of Remark~\ref{rem:bolza}.
	
	\paragraph{Descent of the product}
	In the nonsquare branch we prove that the two labelled elliptic factors live over
	\(\F_{q^{2}}\) and are exchanged by Galois, and that the quotient surface with its
	polarization is defined over \(\Fq\). We do not construct an effective descent
	datum for the product decomposition (Remark~\ref{rem:no-descent}); nothing in the
	paper depends on one.
	
	\paragraph{Cost of the decomposable output}
	The figure \(2\mathrm{M}\) is the cost of detecting the branch. Producing the
	datum costs more: one quadratic-residuosity test for the splitting class and, for
	the explicit models of Theorem~\ref{thm:split}, one square root in \(\Fq\) or
	\(\F_{q^{2}}\) together with a coordinate change of constant cost. See
	Table~\ref{tab:cost}.
	
	\paragraph{Chaining through the quintic stratum}
	A degree-five output has a Weierstrass point at infinity and is not presented as a
	product of three monic quadratics, so before the next step of a chain it requires
	one projective change of coordinate \(x\mapsto1/(x-c)\) with \(c\) outside the
	root set. The certificate then applies to the new displayed splitting without
	modification, since Theorem~\ref{thm:certificate} is total on displayed monic
	triples.
	
	\section{Conclusion}\label{sec:conclusion}
	
	Two questions about a displayed Richelot splitting \(f=uvw\) are often conflated:
	whether the splitting is a valid genus-2 input, and what the geometric type of its
	quotient is. The first is decided by three discriminants and three pairwise
	resultants, the second by the coefficient determinant \(\dcoef\); all seven
	quantities come from data the step already computes, the resultants and the
	determinant being read off the very Bezoutian minors that assemble the brackets,
	so classification is a by-product of the step rather than a pass over it. The
	counting theorems quantify that classification exactly. The parameter space of
	admissible splittings stratifies over the incidence geometry of the discriminant
	parabola, the strata have the closed-form cardinalities of
	Theorem~\ref{thm:counts}, and modulo the affine group the cardinalities become the
	degree-four masses of Theorem~\ref{thm:mass}, with all nontrivial stabilisers
	accounted for by an explicitly located proper sublocus of the decomposable
	stratum. On that stratum, the square class of one resultant, which is an
	invariant of the line joining the three parameter points, splits the configurations
	into two families of asymptotically equal size, one with \(\Fq\)-rational elliptic
	factors and one of Weil-restriction shape with \(\#C(\Fq)=q+1\).
	
	Three directions remain open. The stratification of Galois-permuted splittings,
	for which \S\ref{sec:limits} shows that the branch predicate descends, is
	unresolved. A moduli-level converse to
	Proposition~\ref{prop:extra-involution}, valid over \(\Fq\) and compatible with
	the Galois action on the fifteen splittings of a sextic, would connect \(\dcoef\)
	to Bolza's invariant precisely rather than by analogy. And in characteristic 2 the
	Bezoutian survives but its diagonal becomes a square and the degree analysis
	changes completely; a characteristic-2 analogue requires a separate classification
	theorem rather than a formal reuse of the present one.
	
	\appendix
	
	\section{Three worked examples over \texorpdfstring{\(\F_{101}\)}{F101}}\label{app:examples}
	
	The three examples below form one suite. The first verifies the certificate in the
	generic branch, the second verifies that a valid degree-five output is retained,
	and the third shows that the decomposable branch yields arithmetic data rather
	than a failure. All arithmetic is in \(\F_{101}\).
	
	\begin{example}[A complete genus-2 certificate]\label{ex:A}
		Take \(u=x^{2}+1\), \(v=x^{2}+3x+2\), \(w=x^{2}+5x+3\). The input discriminants
		are \(97,1,13\) and the pairwise resultants are \(\Res(v,w)=3\),
		\(\Res(w,u)=29\), \(\Res(u,v)=10\). The branch determinant is
		\[
		\dcoef=\det\begin{pmatrix}1&0&1\\2&3&1\\3&5&1\end{pmatrix}=-1 .
		\]
		The three brackets are \(U=2x^{2}+2x-1\), \(V=-5x^{2}-4x+5\), \(W=3x^{2}+2x-3\);
		for instance the minor triple of \((v,w)\) is \((2,1,-1)\), giving \(U\). Since
		\(\dcoef\neq0\), Theorem~\ref{thm:certificate} returns
		\(\textsc{Genus2}(C',\phi)\) with \(C'\colon y^{2}=-U(x)V(x)W(x)\). No
		discriminant or resultant calculation on \((U,V,W)\) is needed:
		Proposition~\ref{prop:postcheck} already implies that this product is square-free
		of degree five or six.
	\end{example}
	
	\begin{example}[Aligned but non-decomposable]\label{ex:B}
		Take \(u=x^{2}+1\), \(v=x^{2}+x+2\), \(w=x^{2}+x+4\). The pair \((v,w)\) is
		aligned: its leading minor is \(M_2=0\) and \(U=v'w-vw'=4x+2\) is linear.
		Nevertheless the input discriminants and pairwise resultants are nonzero, while
		\[
		\dcoef=\det\begin{pmatrix}1&0&1\\2&1&1\\4&1&1\end{pmatrix}=-2\neq0 .
		\]
		The output product has degree five, not six, but it is a valid smooth genus-2
		codomain. This is the concrete reason that neither an output-degree test nor a
		guard on the bracket degree may be used as a rejection criterion
		(Remark~\ref{rem:no-guard}).
	\end{example}
	
	\begin{example}\label{ex:C}
		Take \(u=x^{2}+1\), \(v=x^{2}+x+2\), \(w=x^{2}+2x+3=2v-u\). The input is
		square-free and pairwise coprime, and
		\[
		\dcoef=\det\begin{pmatrix}1&0&1\\2&1&1\\3&2&1\end{pmatrix}=0 ,
		\]
		so the three factors form one pencil. For \((u,v)\) the minors are
		\((M_2,M_1,M_0)=(1,1,-1)\) and \(u'v-uv'=x^{2}+2x-1\), so the common deck
		involution is \(\iota(x)=(-x+1)/(x+1)\) and its fixed points satisfy
		\(x^{2}+2x-1=0\), that is \(x=-1\pm\sqrt2\). The splitting class is
		\(R=\Res(u,v)=M_1^{2}-M_2M_0=2\). The three parameter points \((0,1),(1,2),(2,3)\)
		lie on the line \(b=a+1\), with \(m=1\) and \(c=1\), and
		\(m^{2}+c=2\), in agreement with Proposition~\ref{prop:class-line}. Since
		\(101\equiv5\pmod 8\), the class of 2 is a nonsquare, so the line is external to
		\(\Pi\), the two fixed points are conjugate over \(\F_{101^{2}}\), and the
		labelled elliptic factors are exchanged by Galois. The evaluator returns
		\(\textsc{Decomposable}(\mathcal D)\), not an invalid input and not a genus-2
		equation obtained by multiplying the proportional brackets. By
		Corollary~\ref{cor:pointcount} the curve satisfies \(\#C(\F_{101})=102\), which is
		an implementation check requiring no elliptic data; by
		Corollary~\ref{cor:weil}(b) the Weil polynomial of \(\Jac(C)\) is
		\(1-aT^{2}+101^{2}T^{4}\) with \(a=134\), the common trace of the two
		Galois-conjugate factors over \(\F_{101^{2}}\).
	\end{example}
	
	\section{Two sampling measures on the decomposable branch}\label{app:superspecial}
	
	This appendix compares the density of the decomposable branch in the parameter
	space of displayed splittings with its frequency in the setting where the Richelot
	step is most often applied. Unlike the rest of the paper, the comparison is not
	self-contained: the second density is assembled from results quoted from the
	literature, and we state it as such. Nothing in
	\S\S\ref{sec:certificate}--\ref{sec:split} depends on this appendix.
	
	Corollary~\ref{cor:densities} measures \(\Dec\) against \(\Adm\) with respect to
	the uniform measure on the parameter space of displayed splittings over \(\Fq\).
	The motivating application measures something else. There one fixes a superspecial
	principally polarized abelian surface over \(\F_{p^{2}}\) and looks at the fifteen
	maximal 2-Weil-isotropic subgroups of its 2-torsion; the relevant graph is the
	superspecial Richelot isogeny graph, whose vertices are geometric isomorphism
	classes and which is 15-regular counting weights
	\cite[\S5]{FloritSmith2022auto}; this is the setting of the genus-2 isogeny
	problem \cite{CostelloSmith2020}. The two measures differ, and not by a constant.
	
	\begin{proposition}\label{prop:two-densities}
		Let \(q\) be odd.
		\begin{enumerate}[label=\textup{(\roman*)},leftmargin=*]
			\item Among square-free split inputs over \(\Fq\), sampled uniformly, the
			proportion with \(\dcoef=0\) is
			\((q^{2}-3q+3)/\bigl((q-2)(q^{2}-q+4)\bigr)=1/q+O(q^{-3})\).
			\item Let \(C\) be a superspecial genus-2 curve over \(\overline{\F}_p\) whose
			reduced automorphism group is trivial. Then none of the fifteen Richelot quotients
			of \(\Jac(C)\) is a product of elliptic curves.
		\end{enumerate}
	\end{proposition}
	
	\begin{proof}
		(i) is Corollary~\ref{cor:densities}. For (ii), a decomposed Richelot isogeny
		outgoing from \(\Jac(C)\) exists if and only if the reduced automorphism group of
		\(C\) contains an element of order two
		\cite[Prop.~4.3]{KatsuraTakashima2020}; a trivial group contains none.
	\end{proof}
	
	Part (ii) is not a small correction to part (i); it is a different order of
	magnitude. We quote two counts from Florit and Smith
	\cite[Table~2]{FloritSmith2022auto}: the vertices with trivial reduced
	automorphism group, there called Type-A, number
	\(\tfrac1{2880}p^{3}+O(p^{2})\), and the vertices carrying exactly one reduced
	involution, there called Type-I, number \(\tfrac1{48}p^{2}+O(p)\). A Type-A vertex
	has fifteen weight-one edges, none of them to an elliptic product
	\cite[Table~1]{FloritSmith2022auto}, and a generic Type-I vertex has exactly one
	weight-one edge to an elliptic product among its fifteen
	\cite[Prop.~6.1, Case~(1)]{KatsuraTakashima2020}. The remaining vertex types with an involution
	number \(O(p)\) and therefore do not affect the leading term. Hence the proportion
	of decomposable edges among all Richelot edges outgoing from Jacobian vertices is
	\[
	\frac{\tfrac1{48}p^{2}}{15\cdot\tfrac1{2880}p^{3}}+O(p^{-2})=\frac4p+O(p^{-2}),
	\]
	whereas the parameter-space density of
	Proposition~\ref{prop:two-densities}(i), evaluated at \(q=p^{2}\), predicts
	\(p^{-2}\). The two disagree by a factor of order \(p\).
	
	The reason is structural rather than numerical. On the parameter space,
	\(\dcoef=0\) is a codimension-one condition and behaves like one: a hypersurface
	carries a \(1/q\) fraction of the points. On the superspecial locus it is not a
	probabilistic event at all but a consequence of an extra automorphism, and extra
	automorphisms occur on a locus of codimension one in a three-dimensional moduli
	space, hence on \(O(p^{2})\) of the \(O(p^{3})\) vertices. The decomposable branch is arithmetically rigid on the superspecial locus and generic on the parameter space.
	
	This has a concrete consequence for implementations. The figure
	\(1/q+O(q^{-3})\) for the decomposable stratum is the right budget for an
	evaluator fed uniformly random split inputs, for instance in randomized testing
	or in experiments on split Jacobians. It is the wrong budget for a graph walk on
	the superspecial locus over \(\F_{p^{2}}\), where the branch occurs at a
	frequency of order \(1/p\) rather than \(1/p^{2}\): an implementation that
	provisions for it on the basis of (i) underestimates it by a factor of order
	\(p\), and one that omits the branch entirely fails exactly at the
	automorphism-carrying vertices, which is where a graph-walk algorithm is most
	likely to be probed.
	
	Two caveats. The comparison is at the level of vertex counts by reduced
	automorphism type and does not incorporate the automorphism weights with which the
	superspecial Richelot graph is usually equipped
	\cite{FloritSmith2022auto,KatsuraTakashima2020}; a weighted version would require
	the mass of each type rather than its cardinality. And the two measures are taken
	on different objects, \(\Fq\)-rational displayed splittings on one side and
	geometric isomorphism classes of superspecial surfaces on the other, so the
	comparison should be read as a statement about where the branch is concentrated,
	not as an equality of two estimates of one quantity.
	
	\bibliographystyle{elsarticle-num}
	\bibliography{references}
	
\end{document}